\documentclass[a4paper,12pt,reqno]{amsart}

\usepackage{amsthm}
\usepackage{amsmath}
\usepackage{amssymb}
\usepackage{latexsym}
\usepackage{enumerate}
\usepackage{cite}
\usepackage{color}
\usepackage{stmaryrd}

\newtheorem{theorem}{Theorem}[section]
\newtheorem{rem}[theorem]{Remark}
\newtheorem{cor}[theorem]{Corollary}
\newtheorem{prop}[theorem]{Proposition}
\newtheorem{lem}[theorem]{Lemma} 
\makeatletter
\@addtoreset{equation}{section}

\makeatother
\newcommand{\N}{\mathbb{N}}
\newcommand{\Z}{\mathbb{N}_0}
\newcommand{\R}{\mathbb{R}}
\newcommand{\C}{\mathbb{C}}    
\newcommand{\Sc}{\mathcal{S}}         
 
\newcommand{\vp}{\varphi}
\newcommand{\tp}{t^\prime}
\newcommand{\Pa}{\partial}
\newcommand{\Pt}{\partial_t}

\newcommand{\Px}{\partial_x}
\newcommand{\bra}[1]{ \left<  #1 \right>  }
\newcommand{\braa}[1]{\left(  #1 \right)  }
\newcommand{\brab}[1]{\left\{ #1 \right\} }
\newcommand{\brac}[1]{\left[  #1 \right]  }
\newcommand{\abs}[1]{ \left|  #1 \right|  }
\newcommand{\nr}[1]{  \left\| #1 \right\| }

\newcommand{\z}{\mathbf{z}}
\newcommand{\w}{\mathbf{w}}
\newcommand{\uu}{\mathbf{u}}
\newcommand{\vv}{\mathbf{v}}
\newcommand{\f}{\mathbf{f}}
\newcommand{\g}{\mathbf{g}}

\newcommand{\e}{\mathbf{e}}
\newcommand{\A}{\mathcal{A}}
\newcommand{\B}{\mathcal{B}}
\newcommand{\mZ}{\mathcal{Z}}
\newcommand{\M}{\mathcal{M}}
\newcommand{\n}{\mathcal{N}}
\newcommand{\cc}{\mathcal{C}}
\newcommand{\nn}{{\mbox{\boldmath $\mathcal{N}$}}}
\newcommand{\s}{{\mbox{\boldmath $S$}}}
\newcommand{\vd}{\varDelta}
\newcommand{\vvd}{{\mbox{\boldmath $\varDelta$}}}
\newcommand{\Li}{L^\infty}

\newcommand{\Gs}{\Gamma}
\newcommand{\mh}{\mathcal{H}}
\newcommand{\oi}{\omega^{-1}}

\newcommand{\wW}{\widetilde{W}}
\newcommand{\ap}{\alpha^\prime}
\newcommand{\app}{\alpha^{\prime\prime}}
\newcommand{\Ajk}[2]{ (A_{#1}\text{-}#2)}
\newcommand{\ssum}{
\sum_{
\substack{
\alpha(0)+\cdots+\alpha(|\beta|)=\alpha
\\
|\alpha(0)|\ge (5-|\beta|)_+
\\
|\alpha(1)|,\cdots,|\alpha(|\beta|)| \ge 5
}
} 
}
\newcommand{\sssum}{
\sum_{
\alpha(0)+\cdots+\alpha(|\beta|)=\alpha
} 
}
\newcommand{\ssssum}{
\sum_{
\substack{
\alpha(0)+\cdots+\alpha(|\beta|)=\alpha
\\
\text{%
$|\alpha(0)|< (5-|\beta|)_+$ or 
$\exists k\in \{ 1,\cdots,|\beta| \}$ s.t. 
$|\alpha(k)|<5$
}
}
} 
}
\makeatletter 
\newcommand{\bbb}[1]{\bBigg@{#1}} 
\makeatother

\newcommand{\Miijj}[4]{
\begin{pmatrix}
   #1 & #2 \\
   #3 & #4
\end{pmatrix}}
\newcommand{\Miijjj}[6]{
\begin{pmatrix}
   #1 & #2 & #3 \\
   #4 & #5 & #6 
\end{pmatrix}}

\newcommand{\Miii}[3]{
\begin{pmatrix}
   #1 \\
   #2 \\
   #3  
\end{pmatrix}}
\newcommand{\brad}[1]{%
\llparenthesis	
 \,  #1  \, 
\rrparenthesis
}
\DeclareMathOperator*{\supp}{supp}

\begin{document}
\title[%
Inverse scattering problems 
for NLD]{%
On inverse scattering  
for the one-dimensional nonlinear Dirac equation}

\author[H. Sasaki]{%
Hironobu Sasaki$^*$\\
Department of Mathematics and Informatics, 
Chiba University, 263--8522, Japan.}

\thanks{$^*$Supported by 
JSPS KAKENHI Grant Number 22K03367.}

\thanks{$^*$email:
\texttt{sasaki@math.s.chiba-u.ac.jp}} 


\maketitle

\begin{abstract}
The inverse scattering problem 
for the one-dimensional nonlinear Dirac equation 
\begin{align*}
\begin{cases}
i(u_t+u_x)+v=\mathcal{N}_1(u,v),\\
i(v_t-v_x)+u=\mathcal{N}_2(u,v) 
\end{cases}
\end{align*} 
is studied.
We assume that the unknown nonlinearities
$\mathcal{N}_j(u,v)$ ($j=1,2$) of the equation  
belong to 
$C^\infty(\mathbb{C}^2;\mathbb{C})$ and satisfy 
$(\partial_\mathbf{z}^\alpha \mathcal{N}_j)(\mathbf{z})=O(|\mathbf{z}|^{\max\{ 5-|\alpha|,0 \}})$ 
($|\mathbf{z}|\to 0$) 
for any multi-index 
$\alpha \in (\mathbb{N} \cup \{ 0 \})^4$.
Here,
$\mathbf{z}=(z_1,z_2)\in \mathbb{C}^2$,
$
\partial_\mathbf{z}^\alpha=
(\partial_{z_1},\partial_{\overline{z_1}},\partial_{z_2},\partial_{\overline{z_2}})^\alpha,
$
and 
$\partial_{z_1},\partial_{\overline{z_1}},\partial_{z_2},\partial_{\overline{z_2}}$ 
are the Wirtinger differential operators.
Under some additional assumptions on $\mathcal{N}_j$,  
we establish a reconstruction formula 
for $(\partial_\mathbf{z}^\alpha \mathcal{N}_j)(0)$ ($|\alpha|\ge 5$) 
using the knowledge of the scattering operator for the equation.
Although such results were established for  
the nonlinear Schr\"odinger equation \cite{S-2024-ISP}
and 
the nonlinear Klein-Gordon equation \cite{S-2025-ISP}, 
the methods in \cite{S-2024-ISP,S-2025-ISP} 
are not directly applicable to 
the nonlinear Dirac equation 
due to its system nature.
Specifically, 
it is quite difficult to prove the invertibility 
of a matrix associated with the input data.
To overcome this difficulty,
we introduce a new method 
based on the massless limit
for the free solutions.
\end{abstract}

\section{Introduction}
We consider 
the inverse scattering problem 
for the nonlinear Dirac equation 
\begin{align}\label{NLD}
\begin{cases}
i(u_t+u_x)+v=\mathcal{N}_1(u,v),\\
i(v_t-v_x)+u=\mathcal{N}_2(u,v) 
\end{cases} 
\end{align} 
in one space dimension. 
Here, 
$i=\sqrt{-1}$, 
$u=u(t,x)$ and $v=v(t,x)$
are $\C$-valued unknown functions 
defined on $\R \times \R$,
and $\n_1(u,v)$ and $\n_2(u,v)$ are 
the (unknown) nonlinearities.

In order to state hypotheses on $\n_1$ and $\n_2$ precisely,
we introduce some notation.
Put $\Z=\N \cup \{ 0 \}$.
For a Banach space $\mathcal{X}$, 
and $n\in \Z$, 
we denote by $C^n(\C^2;\mathcal{X})$ 
the set of all $\mathcal{X}$-valued functions 
defined on $\C^2$  
and 
$n$-times continuously differentiable 
in the sense of real analysis 
($\C^2 \simeq \mathbb{R}^4$).
We also put 
$C^\infty(\C^2;\mathcal{X})=\bigcap_{n=1}^\infty C^n(\C^2;\mathcal{X})$.
For $j=1,2$,
$h\in C^1(\C^2;\mathcal{X})$, 
and 
$\z=(z_1,z_2)=(x_1+ iy_1 , x_2 + iy_2) \in \C^2$ 
($x_1,x_2,y_1,y_2 \in \R$), 
we define 
the Wirtinger partial derivatives of $h$ at the point $\z$ by 
\begin{align*}
\frac{\Pa}{\Pa z_j}
h(\z)
=
\frac{1}{2}\braa{
\frac{\Pa}{\Pa x_j}
-i
\frac{\Pa}{\Pa y_j}
}h(\z)
\quad \text{and} \quad 
\frac{\Pa}{\Pa \overline{z_j}}
h(\z)
=
\frac{1}{2}\braa{
\frac{\Pa}{\Pa x_j}
+i
\frac{\Pa}{\Pa y_j}
}h(\z).
\end{align*} 
For a multi-index 
$\alpha=(\alpha_1,\alpha_2,\alpha_3,\alpha_4)\in \Z^4$ and 
$h\in C^{|\alpha|}(\C^2;\mathcal{X})$,
we set
\begin{align*}
\Pa_{\z}^\alpha
=
\braa{ \frac{\Pa}{\Pa z_1} }^{\alpha_1}
\braa{ \frac{\Pa}{\Pa \overline{z_1}} }^{\alpha_2}
\braa{ \frac{\Pa}{\Pa z_2} }^{\alpha_3}
\braa{ \frac{\Pa}{\Pa \overline{z_2}} }^{\alpha_4}
\end{align*} 
and 
$
h^{(\alpha)}
=
h^{(\alpha_1,\alpha_2,\alpha_3,\alpha_4)}
=
\Pa_{\z}^\alpha h
$.
For $x\in \R$, 
we put $x_+=\max\{ 0,x \}$.

Throughout this paper, 
we assume that   
$\n_1,\n_2\in C^\infty(\C^2;\C)$.
We also define the hypotheses on $\n_1$ and $\n_2$ 
as follows: 
\begin{enumerate}[(H1)]
  \item 
For any
$j=1,2$ and $\alpha\in \Z^4$, 
$
\n_j^{(\alpha)}(\z)=O\braa{
|\z|^{(5-|\alpha|)_+}
}
$
as 
$|\z|\to 0$.
  \item
There exists a polynomial 
$P(y_1,y_2,y_3,y_4)$ with real coefficients such that 
for any $\z=(z_1,z_2)\in \C^2$, 
\begin{itemize}
  \item 
$\Pa_2 P(z_1,\overline{z_1},z_2,\overline{z_2})=\n_1(\z)$
and 
$\Pa_4 P(z_1,\overline{z_1},z_2,\overline{z_2})=\n_2(\z)$.
  \item 
$
P(e^{i\theta}z_1,\overline{e^{i\theta}z_1},e^{i\theta}z_2,\overline{e^{i\theta}z_2})
=
P(z_1,\overline{z_1},z_2,\overline{z_2})
$
($\theta\in \R$).
  \item 
$P(z_1,\overline{z_1},z_2,\overline{z_2})=P(z_2,\overline{z_2},z_1,\overline{z_1})$.
\end{itemize}
\end{enumerate}
There are many papers 
on the one-dimensional nonlinear Dirac equation
(see, e.g.,  
\cite{%
BZ-2012-NLD,
CP-2006-NLD,
HS-2016-NLD,
HMM-2026-NLD,
H-2011-NLD,
KCDS-pre-NLD,
MNT-2010-NLD,
P-2011-NLD,
PS-2012-NLD,
ZZ-2015-NLD
} 
and references therein).
Most of these papers assume (H2). 
Although (H2) is natural from a physical perspective
and has been utilized in the study of various problems,  
if we consider only the small data scattering,  
then the hypothesis is not needed.
More precisely,
even if $\n_1$ and $\n_2$ satisfy only (H1), 
the scattering operator $\s$ 
for Equation (\ref{NLD}) is well-defined 
on a neighborhood of the origin 
in the Sobolev space $X:=H^1(\R;\C^2)$.
For more details,
see Proposition \ref{prop:S} below.

The aim of this paper 
is to identify the unknown nonlinearities $\n_1$ and $\n_2$ 
from the knowledge of $\s$ under (H1).
In the present paper, 
under some assumptions strictly weaker than (H2),
we establish a reconstruction formula 
for $\n_1^{(\alpha)}(0)$ and $\n_2^{(\alpha)}(0)$ 
for all $\alpha\in \Z^4$ with $|\alpha|\ge 5$.
We remark that it is clear from (H1) that 
$\n_1^{(\alpha)}(0)=\n_2^{(\alpha)}(0)=0$ if $|\alpha|<5$.

\subsection{Direct scattering}
In order to state 
the existence theorem for $\s$, 
we introduce some notation.
For $p\in [1,\infty]$, $s \in \R$, and $d=1,2$, 
we denote the Lebesgue space $L^p(\R;\C^d)$ 
and its norm by $L^p$ and $\nr{\cdot}_p$, 
respectively.
Similarly, 
the Sobolev space $H^s(\R;\C)$ 
and the inhomogeneous Besov space $B_{p,2}^s(\R;\C^d)$ 
are denoted by $H^s$ and $B_p^s$, 
respectively.
For the definition of Besov spaces, 
see, e.g., \cite{BL}.
Put 
$\bra{x}=\sqrt{1+|x|^2}$ ($x\in \R$) 
and 
$\omega=\sqrt{1-\Px^2}$,
where $\Px=\Pa/\Pa x$.
For $x,y\in \R$,
we set 
$x\vee   y=\max\{ x,y \}$ 
and
$x\wedge y=\min\{ x,y \}$. 
Let $\nr{\cdot}$ be 
the norm of $H^1$ defined by 
\begin{align*}
\nr{\vp}
=
\nr{\Px \vp}_2 \vee \nr{    \vp}_2,
\quad
\vp \in H^1.
\end{align*} 
Furthermore, 
the norm of $X$ is given by 
\begin{align*}
\nr{\f}_X
=
\nr{f_1} \vee \nr{f_2},
\quad
\f=\binom{f_1}{f_2} \in X.
\end{align*} 
For a Banach space $\mathcal{X}$, 
$p\in [1,\infty]$, 
and $k\in \Z$, 
we define $L^p \mathcal{X} =L^p(\R;\mathcal{X})$, 
and $C^k \mathcal{X}$ denotes the set of all functions  
which are defined on $\R$ 
and $k$-th order continuously differentiable 
in the $\mathcal{X}$-sense.
Moreover, 
we put the closed ball 
\begin{align*}
B_\eta \mathcal{X} =\brab{
f\in \mathcal{X}; \nr{f}_{\mathcal{X}} \le \eta
},
\quad
\eta>0
\end{align*}  
and the function space 
\begin{align*}
C_b^k \mathcal{X} = \brab{
f\in C^k \mathcal{X}; 
\frac{\Pa^\ell}{\Pa t^\ell} f \in \Li \mathcal{X}, 
\ \ell=0,\cdots,k
}.
\end{align*} 
Define 
$Y=L^4 B^{3/4}_{\infty}(\R;\C^2)$ 
and 
$Z=C_b^0 X \cap Y$.
Furthermore, 
we give 
\begin{align*}
\nr{\uu}_Z
=
     \nr{   \uu}_{L^\infty X} 
\vee \nr{   \uu}_{Y},
\quad
\uu\in Z. 
\end{align*} 
Set 
\begin{align*}
\mh = \Miijj{
-i\Px
}{
-1
}{
-1
}{
 i\Px
}
\quad\text{and}\quad
\nn(\z)
=
\binom{\n_1(\z)}{\n_2(\z)}
\ (\z\in \C^2).
\end{align*} 
Then (\ref{NLD}) is equivalent to 
\begin{align}\label{NLD-1}
i \Pt \uu - \mh \uu = \nn(\uu),
\quad
\uu=\binom{u}{v},
\end{align} 
where $\Pt=\Pa/\Pa t$.
We also obtain the integral form 
\begin{align*}
\uu(t)=e^{-it\mh}\f_0-i\int_0^t e^{-i(t-\tp)\mh} \nn(\uu(\tp))d\tp,
\quad
t\in \R
\end{align*} 
of Equation (\ref{NLD-1}) 
if $\uu(0)=\f_0$.
Note that we have   
\begin{align}\label{id:U(t)}
e^{-it\mh}
=
\Miijj{
\cos(t\omega)-\oi\sin(t\omega)\Px
}{
i\oi\sin(t\omega)
}{
i\oi\sin(t\omega)
}{
\cos(t\omega)+\oi\sin(t\omega)\Px
},
\quad t\in \R.
\end{align} 
For $x,y \ge 0$ and 
for parameters 
$p_1,\cdots,p_m$,
we write $x\lesssim y$ 
(resp. $x\lesssim_{p_1,\cdots,p_m} y$)
if we have 
the inequality $x \le C y$ 
with a positive constant $C$ 
depending only on  the nonlinearity $\nn$
(resp. $\nn,p_1,\cdots,p_m$).

We are ready to state the existence theorem for $\s$.
\begin{prop}\label{prop:S}
Assume (H1).
There exists some $\eta_0\in (0,1]$ such that 
for any $\f_- \in B_{\eta_0}X$ 
there exist a unique time-global solution $\uu \in Z$ 
to the integral equation of the form 
\begin{align}\label{ieq:-}
\uu(t)=
 e^{-it\mh}\f_-
-i\int_{-\infty}^t e^{-i(t-\tp)\mh} \nn(\uu(\tp))d\tp,
\quad
t\in \R,
\end{align} 
and a unique function $\f_+ \in X$, 
which satisfies 
\begin{align}\label{id:S-op}
\f_+=\f_- - i\int_{\R} e^{it\mh}\nn(\uu(t))dt,
\end{align} 
such that $\nr{\nn(\uu)}_{L^1 X} \lesssim \nr{\f_-}_X^5$ 
and 
\begin{align*}
\lim_{t\to \pm \infty}
\nr{\uu(t)-e^{-it\mh}\f_\pm}_X=0.
\end{align*} 
Hence the scattering operator 
$\s:B_{\eta_0}X \ni \f_- \mapsto \f_+ \in X$ 
for (\ref{NLD}) is well-defined.
\end{prop}

We give an outline of the proof of Proposition \ref{prop:S}  
in Section \ref{sec:2} below.

\subsection{Inverse scattering}
The inverse scattering problem (ISP)
for nonlinear dispersive equations is 
to identify unknown nonlinearities 
under the assumption that 
the scattering operator is a given mapping 
(see, e.g. \cite{%
BUY-2022-ISP,
CM-2024-ISP,
CM-2024-2-ISP,
HMG-2023-ISP,
HKV-2025-ISP,
KMV-2023-ISP,
MSSSS-2018-ISP,
MS-1973-ISP,
S-2007-ISP,
S-2008-ISP,
S-2012-ISP,
S-2024-ISP,
S-2025-ISP,
SS-2011-ISP,
S-1974-ISP,
W-1997-ISP,
W-2000-ISP,
W-2000-2-ISP,
W-2001-ISP,
W-2001-2-ISP,
W-2002-ISP,
W-2005-1-ISP,
W-2005-2-ISP
} 
and references therein).
Recently, 
the author studied the ISPs 
for the nonlinear Schr\"odinger equation (NLS) \cite{S-2024-ISP} 
and the nonlinear Klein-Gordon equation (NLKG) \cite{S-2025-ISP}.
For instance, 
in \cite{S-2024-ISP}, 
the ISP for the NLS
\begin{align*}
iu_t+\Delta u=F(u)
\end{align*} 
in two space dimensions was studied,
where $F=F(z)$ is an unknown function in $C^\infty(\C;\C)$ 
and satisfies 
\begin{align*}
F^{(n_1,n_2)}(z):=
\braa{\frac{\Pa}{\Pa z}}^{n_1}
\braa{\frac{\Pa}{\Pa \overline{z}}}^{n_2}F(z)
=
O\braa{
|z|^{(3-n_1-n_2)_+}
}
\end{align*} 
as $|z|\to 0$ ($n_1,n_2\in \Z$).
It was shown that 
by using the knowledge of scattering states, 
one has a reconstruction formula for 
\begin{align*}
\sum_{n=0}^N\binom{N}{n}
F^{(N-n,n)}(0) 
\int_{\R\times \R^2}
\braa{            e^{it\Delta}\phi   }^{N-n}
\braa{ \overline{ e^{it\Delta}\phi } }^{n+1}
d(t,x)
\end{align*} 
for any $N\in \Z$ with $N\ge 3$
if $F^{(n_1,n_2)}(0)$ is known   
for any $n_1,n_2\in \Z$ with $n_1+n_2<N$.
Here, 
$\phi$ is an arbitrary function in $H^1(\R^2;\C)$.
Since it is easy to choose functions 
$\phi_0,\phi_1,\cdots,\phi_N \in H^1(\R^2;\C)$ 
so that 
the $(N+1) \times (N+1)$ matrix
\begin{align*}
\braa{
\int_{\R \times \R^2}
\braa{           e^{it\Delta}\phi_\ell  }^{N-n}
\braa{\overline{ e^{it\Delta}\phi_\ell }}^{n+1}
d(t,x)
}_{0\le \ell,n \le N}
\end{align*} 
is invertible,
one can uniquely determine the exact value 
of each $F^{(N-n,n)}(0)$ for all $n=0,1,\cdots,N$.
As a main result in \cite{S-2025-ISP} 
one has a reconstruction formula
of $F^{(n_1,n_2)}(0)$ for any $n_1,n_2\in \Z$ with $n_1+n_2\ge 3$.
\\

In the present paper,
we give a reconstruction formula  
for $\nn^{(\alpha)}(0)$ ($\alpha\in \Z^4$)
under some hypotheses on $\nn$.
We always assume (H1) 
in order to use the knowledge of the scattering operator $\s$ 
defined in Proposition \ref{prop:S}.
Therefore, 
it is sufficient to consider the case $N:=|\alpha|\ge 5$.
We remark that 
it seems quite difficult to identify $\nn^{(\alpha)}(0)$ 
by only using the method in \cite{S-2024-ISP,S-2025-ISP}. 
In order to explain the difficulty, 
we now introduce some notation.
For $\z=(z_1,z_2) \in \C^2$,   
we define the four-dimensional vector
$\brad{\z}=\brad{z_1,z_2}=(z_1,\overline{z_1},z_2,\overline{z_2})$.
We denote  
the set $\{ \alpha\in \Z^4 \,;\, |\alpha|=N\}$ 
and its cardinality 
by $S_N$ and $L_N$, 
respectively.
We also fix a bijection $A:\{ n\in \N \,;\, n\le L_N \} \to S_N$. 
If $\nn^{(\beta)}(0)$ with $|\beta|<N$ is known, 
then by the method 
in \cite{S-2024-ISP,S-2025-ISP}  
we see the exact value of 
\begin{align}
&
\sum_{k=1}^{L_N}
\frac{N!}{A(k)!}
\int_{\R \times \R}
\brad{ 
e^{-it\mh}\f
}^{A(k)}
\brad{ e^{-it\mh}\g }^{(0,1,0,0)}
d(t,x)
\,\n_1^{(A(k))}(0)
\nonumber
\\ 
&\quad +
\sum_{k=1}^{L_N}
\frac{N!}{A(k)!}
\int_{\R \times \R}
\brad{ 
e^{-it\mh}\f
}^{A(k)}
\brad{ e^{-it\mh}\g }^{(0,0,0,1)}
d(t,x)
\,\n_2^{(A(k))}(0)
\nonumber
\\ 
&=:
\sum_{k=1}^{L_N}
\frac{N!}{A(k)!}
I_{k,1}[\f,\g] \, \n_1^{(A(k))}(0)
+
\sum_{k=1}^{L_N}
\frac{N!}{A(k)!}
I_{k,2}[\f,\g] \, \n_2^{(A(k))}(0),
\label{II}
\end{align} 
where $\f$ and $\g$ are arbitrary functions in $X$.
If there exist some functions 
$\f_{\ell,j}$ and $\g_{\ell,j}$ 
($\ell=1,\cdots,L_N$ and $j=1,2$) 
such that 
the $2L_N \times 2L_N$ matrix 
\begin{align}\label{matrix}
\braa{
\Miijj{
I_{k,1}[\f_{\ell,1},\g_{\ell,1}]
}{
I_{k,2}[\f_{\ell,1},\g_{\ell,1}]
}{
I_{k,1}[\f_{\ell,2},\g_{\ell,2}]
}{
I_{k,2}[\f_{\ell,2},\g_{\ell,2}]
}
}_{1\le k,\ell \le L_N}
\end{align} 
is invertible, 
then we can identify 
$\nn^{(\alpha)}(0)$ for any $\alpha \in S_N$.
Unfortunately, 
unlike the NLS case mentioned above, 
it seems that the structure of (\ref{matrix}) is too complicated 
to find such functions $\f_{\ell,j}$ and $\g_{\ell,j}$.
In other words, 
the method in \cite{S-2024-ISP,S-2025-ISP} 
is not directly applicable to 
the nonlinear Dirac equation 
due to its system nature. 
In order to avoid this complication, 
we apply the \textbf{massless limit}, 
which indicates that 
the solution $\uu_m$ to the massive equation $i\Pt \uu - \mh_m \uu=0$  
approaches 
the solution $\uu_0$ to the massless equation $i\Pt \uu - \mh_0 \uu=0$ 
as $m\to +0$ in some sense. 
Here 
\begin{align*}
\mh_m=\Miijj{-i\Px}{-m}{-m}{i\Px}, \quad m\ge 0.
\end{align*} 
Note that the massless solution $\uu_0$ is expressed 
by the following decoupled form: 
\begin{align*}
\uu_0 = \binom{f_1(x-t)}{f_2(x+t)} 
\quad\text{when}\quad 
\uu_0(0,x)=\binom{f_1(x)}{f_2(x)} \in X. 
\end{align*} 
Therefore, 
if the massless limit is applicable to (\ref{matrix}), 
then the matrix would take a simple form. 
Although the massless solution $\uu_0$, 
which is a pair of right- and left-traveling waves, 
does not decay as $t\to \pm \infty$,  
we focus on the integrability of $f_1(x-t) f_2(x+t)$ 
over $\R \times \R$.  
As a consequence, 
in the present paper, 
we assume Hypotheses (H3)--(H5) described precisely below, 
derive an alternative relation instead of (\ref{II}),  
and establish reconstruction formulas 
for $\nn^{(\alpha)}(0)$ ($|\alpha|\ge 5$).

\subsection{Main results}\label{sebsec:main}
To state our main results precisely,
we introduce some notation and definitions. 
By $\Sc$,
we denote 
the Schwartz space of all \textbf{real-valued} rapidly decaying functions on $\R$.
We define the inner product on $L^2(\R;\C^2)$ 
by 
\begin{align*}
\bra{\f,\g}
=
\int_{\R}
\braa{
f_1(x) \overline{g_1(x)}
+
f_2(x) \overline{g_2(x)}
}
dx,
\quad
\f=\binom{f_1}{f_2},
\g=\binom{g_1}{g_2} \in L^2(\R;\C^2).
\end{align*} 
Put vectors $\e_1=(1,0)$ and $\e_2=(0,1)$.
For $\vp\in H^1$, 
we set 
\begin{align*}
(D_m \vp)(x)=\vp(x/m) 
\ (m>0,x\in \R)
\quad\text{and}\quad
\mu(\z)\vp=\binom{z_1 \vp}{z_2 \vp}
\ (\z= (z_1,z_2)\in \C^2).
\end{align*} 
For $w_1,w_2\in \C$,
we define the translation operator $\tau[w_1,w_2]$ by 
\begin{align*}
(\tau[w_1,w_2]h)(z_1,z_2)
=
h(z_1+w_1,z_2+w_2),
\quad
h\in C^0(\C^2;\C), \ 
z_1,z_2\in \C.
\end{align*} 
For $\lambda\in \R$, 
we put 
\begin{align*}
\vd_{1,\lambda}&=(\tau[\lambda,0]-i\tau[i\lambda,0]-1+i)/2,
\\
\vd_{2,\lambda}&=(\tau[\lambda,0]+i\tau[i\lambda,0]-1-i)/2,
\\
\vd_{3,\lambda}&=(\tau[0,\lambda]-i\tau[0,i\lambda]-1+i)/2,
\\
\vd_{4,\lambda}&=(\tau[0,\lambda]+i\tau[0,i\lambda]-1-i)/2
\end{align*} 
and 
$
\vvd_\lambda^\alpha=
\vd_{1,\lambda}^{\alpha_1}
\vd_{2,\lambda}^{\alpha_2}
\vd_{3,\lambda}^{\alpha_3}
\vd_{4,\lambda}^{\alpha_4}
$
($\alpha=(\alpha_1,\alpha_2,\alpha_3,\alpha_4) \in \Z^4$).
For $m>0$, 
$\vp\in H^1$, 
and $(t,x)\in \R \times \R$, 
we set $\omega_m=\sqrt{m^2-\Px^2}$,
$\vp_+(t,x)=\vp(t-x)$, 
$\vp_-(t,x)=\vp(t+x)$,
\begin{align*}
R_{\vp,m}(t,x)
&=
\braa{\cos (t\omega_m) \vp - \frac{\sin(t\omega_m)}{\omega_m}\Px \vp}(x),
\\
L_{\vp,m}(t,x)
&=
\braa{\cos (t\omega_m) \vp + \frac{\sin(t\omega_m)}{\omega_m}\Px \vp}(x)
\end{align*} 
and 
\begin{align*}
M_{\vp,m}(t,x)
=
m \braa{\frac{\sin(t\omega_m)}{\omega_m}\vp}(x).
\end{align*}  
Furthermore, 
we abbreviate 
$R_{\vp,1}$,
$L_{\vp,1}$,
and $M_{\vp,1}$ 
to
$R_\vp$,
$L_\vp$, and 
$M_\vp$,
respectively.
For $\vp\in H^1$,
$m>0$, 
and 
$\alpha\in \Z^4$ 
with $|\alpha| \ge 5$,
we define 
\begin{align*}
d_{\vp,\alpha,+}
&=
\int_{\R \times \R}
\brad{\vp_+, \vp_- }^{\alpha+(0,1,0,0)}
d(t,x),
\\
d_{\vp,\alpha,-}
&=
\int_{\R \times \R}
\brad{\vp_+, \vp_- }^{\alpha+(0,0,0,1)}
d(t,x),
\\
d_{\vp,\alpha,m,R}
&=
\int_{\R\times \R}  
\braa{ R_{\vp,m} }^{|\alpha|+1}
d(t,x),
\\
d_{\vp,\alpha,m,M}
&=
\int_{\R\times \R}  
\braa{ R_{\vp,m} }^{|\alpha|-1} \braa{ M_{\vp,m} }^2
d(t,x).
\end{align*} 
For $t\in \R$, 
we define the mapping $G$ by  
\begin{align*}
(\Gs G)(t)=-i\int_{-\infty}^t e^{-i(t-\tp)\mh}G(\tp)d\tp,
\quad G \in L^1X.
\end{align*}
For $\alpha=(\alpha_1,\alpha_2,\alpha_3,\alpha_4) \in \Z^4$,
put $\widetilde{\alpha}=(\alpha_3,\alpha_4,\alpha_1,\alpha_2)$ 
and 
let $k_\alpha$ be the largest non-negative integer $k$ 
such that $(5-k)_+ + 5k \le |\alpha|$.
Furthermore,
if $|\alpha|>0$,
then 
for $\ell=1,\cdots,|\alpha|$ 
and $\z=(z_1,z_2)\in \C^2$ we set 
\begin{align*}
\cc_{\alpha,\ell}(\z)=
  \begin{cases}
          z_1        &
(1 \le \ell \le \alpha_1),\\
\overline{z_1}       &
(\alpha_1+1 \le \ell \le \alpha_1+\alpha_2),\\
          z_2        &    
(\alpha_1+\alpha_2+1 \le \ell \le \alpha_1+\alpha_2+\alpha_3),\\
\overline{z_2}       &    
(\alpha_1+\alpha_2+\alpha_3+1 \le \ell \le |\alpha|).
  \end{cases}
\end{align*} 
For $N\in \Z$ with $N\ge 5$ 
and $\alpha=(\alpha_1,\alpha_2,\alpha_3,\alpha_4)\in \Z^4$ with $|\alpha|=N$,
we set the following conditions:
\begin{enumerate}[({A}{$_\text{1}$}-1)]
  \item $\alpha_1+\alpha_2 \ge 2$ and $\alpha_3+\alpha_4 \ge 3$.
  \item $N$ is odd, $\alpha_1+\alpha_2 = N$, and $\alpha_3+\alpha_4 = 0$.
  \item $N$ is odd, $\alpha_1+\alpha_2 = N-2$, and $\alpha_3+\alpha_4 = 2$.
  \item $N$ is odd, $\alpha_1+\alpha_2 = 1$, and $\alpha_3+\alpha_4 = N-1$.
  \item $N$ is odd, and either $\alpha=(n,n,1,0)$ or $\alpha=(n+1,n-1,0,1)$.
  \item $N$ is odd, and $\alpha=(0,0,n+1,n)$.
  \item $N$ is odd, and $\alpha$ satisfies none of (A$_1$-1)--(A$_1$-6).
\end{enumerate}
and 
\begin{enumerate}[({A}{$_\text{2}$}-1)]
  \item $\alpha_1+\alpha_2 \ge 3$ and $\alpha_3+\alpha_4 \ge 2$.
  \item $N$ is odd, $\alpha_1+\alpha_2 = 0$, and $\alpha_3+\alpha_4 = N$.
  \item $N$ is odd, $\alpha_1+\alpha_2 = 2$, and $\alpha_3+\alpha_4 = N-2$.
  \item $N$ is odd, $\alpha_1+\alpha_2 = N-1$, and $\alpha_3+\alpha_4 = 1$.
  \item $N$ is odd, and either $\alpha = (1,0,n,n)$ or $\alpha = (0,1,n+1,n-1)$.
  \item $N$ is odd, and $\alpha = (n+1,n,0,0)$.
  \item $N$ is odd, and $\alpha$ satisfies none of (A$_2$-1)--(A$_2$-6).
\end{enumerate}
Here, 
we have put $n=(N-1)/2$.
For simplicity of notation, 
for $j=1,2$ and $k=1,\cdots,7$, 
we write $\alpha \in   \Ajk{j}{k}$ 
if $\alpha$ satisfies Condition (A$_\text{j}$-k).
Note that 
$\alpha \in \Ajk{1}{7}$ if and only if 
$|\alpha|$ is odd, 
and 
there exist some $n,\ell\in \Z$ satisfying 
one of the following conditions: 
\begin{itemize}
  \item 
$\alpha=(n,\ell,1,0)$ with $n+\ell+1=N$ and $n-\ell  \neq 0$.
  \item 
$\alpha=(n,\ell,0,1)$ with $n+\ell+1=N$ and $n-\ell  \neq 2$. 
  \item 
$\alpha=(0,0,n,\ell)$ with $n+\ell  =N$ and $n-\ell  \neq 1$.
\end{itemize}
Put $\mathbf{0}=(0,0,0,0)$.
For $\vp\in H^1$,
$\alpha,\beta \in \Z^4$,
and a mapping 
$\A:\Z^4 \to \C^2$,
we define 
$V_{\alpha,\beta}^\A[\vp]$, 
$V_\alpha^\A[\vp]$, 
$W_\alpha^\A[\vp]$,
and 
$\wW_\alpha^\A[\vp]$ 
by induction on $|\alpha|$ 
as follows:
\begin{itemize}
\item 
$
V_{\alpha,\beta}^\A[\varphi]
  = 
  \displaystyle\sum\limits_{\gamma\le \alpha}
  \frac{\alpha!}{\gamma! (\alpha-\gamma)!}
  \brad{ R_\varphi,L_\varphi }^{\alpha-\gamma}
  \brad{ iM_\varphi,iM_\varphi }^{\gamma}
  \A(\beta+\alpha-\gamma+\widetilde{\gamma}).
$
  \item 
$V_\alpha^\A[\vp]=V_{\alpha,\mathbf{0}}^\A[\varphi]$.
  \item 
If $|\alpha|<9$, 
then $W_\alpha^\A[\vp]=0$ 
and $\wW_\alpha^\A[\vp]=V_\alpha^\A[\vp]$.
  \item 
If $|\alpha|\ge 9$, 
then 
\begin{align*}
  W_\alpha^\A[\vp]
=
\sum_{1\le |\beta| \le k_\alpha}
\frac{\alpha!}{\beta!}
\ssum
\braa{
\prod_{\ell=1}^{|\beta|}
\frac{
\cc_{\beta,\ell} \braa{ \Gs \wW_{\alpha(\ell)}^\A[\vp] }
}{\alpha(\ell)!}
}
\frac{
V_{\alpha(0),\beta}^\A[\vp]
}{\alpha(0)!}
\end{align*} 
and
$\wW_\alpha^\A[\vp]=V_\alpha^\A[\vp]+W_\alpha^\A[\vp]$.
\end{itemize}
Let $\s$ and $\eta_0$ be the mapping and the number 
appearing in Proposition \ref{prop:S},
respectively.
For $\vp\in H^1$, 
a mapping $\A:\Z^4 \to \C^2$,
$\z\in \C^2$, 
$\lambda,m>0$,
$\alpha=(\alpha_1,\alpha_2,\alpha_3,\alpha_4) \in \Z^4$, 
and 
$j=1,2$
with 
$|\alpha|\ge 5$
and 
$(|\z|+|\alpha|\lambda)\nr{\vp}\le \eta_0$,
we introduce 
\begin{align*}
I_{\alpha,j}^\A [\vp](m,\lambda)
&=
im^{-2-|\alpha|/2}
\brac{
\lambda^{-|\alpha|}
\vvd_\lambda^\alpha
\bra{
\s (\mu(\z) m^{1/2} D_m \vp), \mu(\e_j)D_m \vp
}}_{\z=\binom{\lambda}{\lambda}}
\\
&\quad -
m^{-2-|\alpha|/2}
\int_{\R}
\bra{
W_\alpha^\A[m^{1/2} D_m \vp](t),
e^{-it\mh}\mu(\e_j)D_m \vp
}
dt,
\end{align*} 
and define 
\begin{align*}
&
K_{\alpha,1}^\A[\vp](m,\lambda)
\\
&=
\begin{cases}
\dfrac{
I_{\alpha,1}^\A[\vp](m,\lambda)
}{
d_{\vp,\alpha,+}} 
& \text{if $\alpha\in \Ajk{1}{1}$,} \\[2ex]
\dfrac{
I_{\alpha,1}^\A[\vp](m,\lambda)
}{
d_{\vp,\alpha,m,R}}    
& \text{if $\alpha\in \Ajk{1}{2}$,} \\[2ex] 
\dfrac{
I_{\alpha,1}^\A[\vp](m,\lambda)
}{
d_{\vp,\alpha,+}}    
-(-1)^{(\alpha_3-\alpha_4)/2}
\dfrac{
d_{\vp,\alpha,m,M} 
I_{(\alpha_1+\alpha_3,\alpha_2+\alpha_4,0,0),1}^\A[\vp](m,\lambda)
}{
d_{\vp,\alpha,+}\, d_{\vp,\alpha,m,R}
}
& \text{if $\alpha\in \Ajk{1}{3}$,} \\[2ex] 
\dfrac{
I_{\alpha,1}^\A[\vp](m,\lambda)
}{
d_{\vp,\alpha,+}}    
-(-1)^{(1+\alpha_1-\alpha_2)/2}
\dfrac{
d_{\vp,\alpha,m,M} 
I_{(0,0,\alpha_1+\alpha_3,\alpha_2+\alpha_4),2}^\A[\vp](m,\lambda)
}{
d_{\vp,\alpha,+}\, d_{\vp,\alpha,m,R}
}
& \text{if $\alpha\in \Ajk{1}{4}$,} \\[2ex] 
\dfrac{(2|\alpha|+2)^{-1}}{d_{\vp,\alpha,m,M}}
\braa{
I_{(n,n,1,0),1    }^\A[\vp](m,\lambda)
+
I_{(n+1,n-1,0,1),1}^\A[\vp](m,\lambda)
}
& \text{if $\alpha\in \Ajk{1}{5}$,} \\[2ex] 
\dfrac{(2|\alpha|+2)^{-1}}{d_{\vp,\alpha,m,M}}
\braa{
(|\alpha|+2)I_{(1,0,n,n),2    }^\A[\vp](m,\lambda)
-
|\alpha|    I_{(0,1,n+1,n-1),2}^\A[\vp](m,\lambda)
}
& \text{if $\alpha\in \Ajk{1}{6}$,} \\[2ex]
0
& \text{otherwise}
\end{cases}
\end{align*} 
and 
\begin{align*}
&
K_{\alpha,2}^\A[\vp](m,\lambda)
\\
&=
\begin{cases}
\dfrac{
I_{\alpha,2}^\A[\vp](m,\lambda)
}{
d_{\vp,\alpha,-}} 
& \text{if $\alpha\in \Ajk{2}{1}$,} \\[2ex]
\dfrac{
I_{\alpha,2}^\A[\vp](m,\lambda)
}{
d_{\vp,\alpha,m,R}}    
& \text{if $\alpha\in \Ajk{2}{2}$,} \\[2ex] 
\dfrac{
I_{\alpha,2}^\A[\vp](m,\lambda)
}{
d_{\vp,\alpha,-}}    
-(-1)^{(\alpha_1-\alpha_2)/2}
\dfrac{
d_{\vp,\alpha,m,M} 
I_{(0,0,\alpha_1+\alpha_3,\alpha_2+\alpha_4),2}^\A[\vp](m,\lambda)
}{
d_{\vp,\alpha,-}\, d_{\vp,\alpha,m,R}
}
& \text{if $\alpha\in \Ajk{2}{3}$,} \\[2ex] 
\dfrac{
I_{\alpha,2}^\A[\vp](m,\lambda)
}{
d_{\vp,\alpha,-}}    
-(-1)^{(1+\alpha_3-\alpha_4)/2}
\dfrac{
d_{\vp,\alpha,m,M} 
I_{(\alpha_1+\alpha_3,\alpha_2+\alpha_4,0,0),1}^\A[\vp](m,\lambda)
}{
d_{\vp,\alpha,-}\, d_{\vp,\alpha,m,R}
}
& \text{if $\alpha\in \Ajk{2}{4}$,} \\[2ex] 
\dfrac{(2|\alpha|+2)^{-1}}{d_{\vp,\alpha,m,M}}
\braa{
I_{(1,0,n,n),2    }^\A[\vp](m,\lambda)
+
I_{(0,1,n+1,n-1),2}^\A[\vp](m,\lambda)
}
& \text{if $\alpha\in \Ajk{2}{5}$,} \\[2ex] 
\dfrac{(2|\alpha|+2)^{-1}}{d_{\vp,\alpha,m,M}}
\braa{
(|\alpha|+2)I_{(n,n,1,0),1    }^\A[\vp](m,\lambda)
-
|\alpha|    I_{(n+1,n-1,0,1),1}^\A[\vp](m,\lambda)
}
& \text{if $\alpha\in \Ajk{2}{6}$,} \\[2ex] 
0
& \text{otherwise.}
\end{cases}
\end{align*}
If $\A$ is the mapping 
$\Z^4\ni \alpha \mapsto \nn^{(\alpha)}(0) \in \C^2$,
then we abbreviate 
$V_{\alpha,\beta}^\A[\vp]$,
$V_\alpha^\A[\vp]$,
$W_\alpha^\A[\vp]$,
$\wW_\alpha^\A[\vp]$,
$I_{\alpha,j}^\A[\vp](m,\lambda)$,
and 
$K_{\alpha,j}^\A[\vp](m,\lambda)$
to
$V_{\alpha,\beta}[\vp]$,
$V_\alpha[\vp]$,
$W_\alpha[\vp]$,
$\wW_\alpha[\vp]$,
$I_{\alpha,j}[\vp](m,\lambda)$,
and 
$K_{\alpha,j}[\vp](m,\lambda)$
respectively.
We next define the hypotheses on $\n_1$ and $\n_2$ 
as follows: 
\begin{enumerate}[(H1)]
\setcounter{enumi}{2}
  \item 
For any $n\in \N$ with $n\ge 2$,
we have 
$\n_1^{(n,n,1,0)}(0)=\n_1^{(n+1,n-1,0,1)}(0)$ 
and 
$\n_2^{(1,0,n,n)}(0)=\n_2^{(0,1,n+1,n-1)}(0)$. 
  \item 
If $N\in \N$ is odd,
then $\n_j^{(\alpha)}(0)=0$ 
for any $j=1,2$ and $\alpha \in \Ajk{j}{7}$.
  \item 
For any $\alpha\in \Z^4$, 
if $|\alpha|$ is even, 
then $\nn^{(\alpha)}(0)=0$.
\end{enumerate}

\begin{rem}\label{rem:isp-notation}
We provide some comments on the above notation.
\begin{enumerate} 
  \item 
For any $\vp \in H^1$, 
we obtain $\nr{D_m \vp}\le m^{-1/2}\nr{\vp}$  
$(m\in (0,1))$, 
\begin{align*}
\mu(\e_1)\vp = \binom{\vp}{0}
\quad\text{and}\quad 
\mu(\e_2)\vp = \binom{0}{\vp}
\end{align*}
and we see from (\ref{id:U(t)}) that 
\begin{align}\label{id:U(t):2}
e^{-it\mh}\mu(\z)\vp
=
\binom{
z_1 R_{\vp}(t)+i z_2 M_{\vp}(t)
}{
z_2 L_{\vp}(t)+i z_1 M_{\vp}(t)
},
\quad
\ t\in \R, \ \z=(z_1,z_2)\in \C^2.
\end{align}
  \item 
Mappings 
$\vd_{1,\lambda}, \cdots, \vd_{4,\lambda}$ 
and $\vvd_\lambda^\alpha$
are forward difference operators 
associated with the Wirtinger derivatives 
$\Pa/\Pa z_1$, $\Pa/\Pa \overline{z_1}$, 
$\Pa/\Pa z_2$, $\Pa/\Pa \overline{z_2}$ 
and $\Pa_\z^\alpha$, 
respectively.  
  \item 
We see that 
$R_{D_m \vp}(t,x)=R_{\vp,m}(t/m,x/m)$,
$L_{D_m \vp}(t,x)=L_{\vp,m}(t/m,x/m)$,
and 
$M_{D_m \vp}(t,x)=M_{\vp,m}(t/m,x/m)$ 
for any $\vp \in H^1$, 
$m>0$, 
and $(t,x)\in \R \times \R$. 
  \item
If $\vp$ is even,
then we have 
\begin{align}\label{id:R-L}
d_{\vp,N,m,R}
=\int_{\R\times \R}  
\braa{ L_{\vp,m} }^{N+1}
d(t,x)
\quad\text{and}\quad
d_{\vp,N,m,M}
=
\int_{\R\times \R}  
\braa{ L_{\vp,m} }^{N-1} \braa{ M_{\vp,m} }^2
d(t,x).
\end{align} 
  \item 
The mapping $\Gs$ becomes a bounded linear operator 
from $L^1X$ to $Z$ 
because of the Strichartz type estimate  
(see Proposition \ref{prop:St}). 
  \item 
Let $\lfloor \cdot \rfloor$ be the floor function.
Then we have for any $\alpha \in \Z^4$, 
\begin{align*}
k_\alpha
=
\left\lfloor \frac{|\alpha|-5}{4} \right\rfloor 
\wedge 
\left\lfloor \frac{|\alpha|}{5}   \right\rfloor .
\end{align*}
In particular, 
if $|\alpha|\ge 9$ then $k_\alpha\ge 1$.
  \item
Assume that $\alpha \in \Z^4$ satisfies $|\alpha|\ge 9$.
Let $\A$ be a mapping from $\Z^4$ to $\C^2$.
Then $W_\alpha^\A[\vp]$ depends only on $\vp$, 
$\alpha$ and $\A(\gamma)$ 
with $|\gamma| \le |\alpha|-4$.
Indeed, 
multi-indices 
$\beta,\alpha(0),\alpha(1),\cdots,\alpha(|\beta|)$ 
in the definition of $W_\alpha^\A[\vp]$ satisfy 
\begin{align*}
|\alpha|-\braa{ |\alpha(0)+\beta| }
&=
|\alpha(1)|+\cdots+|\alpha(|\beta|)|-|\beta|
\ge 
5|\beta|-|\beta|\ge 4
\end{align*} 
and
\begin{align*}
|\alpha|-|\alpha(\ell)|
&=
|\alpha(0)|+|\alpha(1)|+\cdots+|\alpha(|\beta|)|-|\alpha(\ell)|
\\
&\ge 
(5-|\beta|)_+ + 5(|\beta|-1)
\ge 
4|\beta| \ge 4
\end{align*}  
for all $\ell=1,\cdots,|\beta|$. 
Hence we see that 
$W_\alpha[\vp]$, 
$I_{\alpha,j}[\vp](m,\lambda)$, 
and  
$K_{\alpha,j}[\vp](m,\lambda)$
become known
provided that 
$\nn^{(\gamma)}(0)$ with $|\gamma|\le |\alpha|-4$ is known, 
and that the knowledge of $\mathbf{S}$ on $B_{\eta_0}X$ is given. 
  \item 
Hypotheses (H3)--(H5) are strictly weaker than (H2).
We have only to show that (H2) implies (H3).
Suppose that (H2) is true.
Then $P^{(\alpha)}(0)$ is a real number 
for any $\alpha\in \Z^4$.
Therefore, 
we obtain 
$P^{(n,n+1,1,0)}(0)=\overline{P^{(n+1,n,0,1)}(0)}=P^{(n+1,n,0,1)}(0)$.
It follows that 
\begin{align*}
 \n_1^{(n,n,1,0)}(0)=P^{(n,n+1,1,0)}(0)=P^{(n+1,n,0,1)}(0)
=\n_1^{(n+1,n-1,0,1)}(0).
\end{align*} 
Similarly, 
we have 
$\n_2^{(1,0,n,n)}(0)=\n_2^{(0,1,n+1,n-1)}(0)$.
\end{enumerate}
\end{rem}

We are ready to state the first main result.
\begin{theorem}\label{thm:main:1}
Fix $N\in \N$ with $N\ge 5$.
Suppose (H1), (H3), 
and 
that $\nn^{(\beta)}(0)$ with $|\beta|\le N-4$ is known.
Choose an even function 
$\vp \in \Sc \cap B_{\eta_0}H^1$ 
so that $\supp\vp \subset [-1,1]$, 
and that $\vp>0$ on $(-1,1)$.
Set $j=1,2$.
Let $\alpha=(\alpha_1,\alpha_2,\alpha_3,\alpha_4)\in \Z^4$ 
satisfy $|\alpha|=N$  
and belong to $\Ajk{j}{k}$ for some $k=1,\cdots,6$.
It follows that for any 
$m\in (0,1/2)$
and 
$0<\lambda <1/(N+1)$, 
\begin{align}\label{est:thm:main:1}
\abs{
K_{\alpha,j}[\vp](m,\lambda)-\n_j^{(\alpha)}(0)
}
\lesssim_{N,\vp}
m^{-2-(N+1)/2}\lambda+m^{1/2}.
\end{align} 
\end{theorem}

\begin{rem}\label{rem:thm:main:1}
\begin{enumerate}[(1)]
  \item
Theorem \ref{thm:main:1} implies that  
if $\alpha \in \Ajk{j}{k}$ for some $j=1,2$ and $k=1,\cdots,6$, 
then 
\begin{align*}
\abs{
K_{\alpha,j}[\vp](\lambda^{2/(N+6)}/2,\lambda)-\n_j^{(\alpha)}(0)
}
\lesssim_{N,\vp}
\lambda^{1/(N+6)}.
\end{align*}
Therefore, 
we have the reconstruction formula for $\n_j^{(\alpha)}(0)$ 
as follows:
\begin{align*}
\lim_{\lambda\to +0}
K_{\alpha,j}[\vp](\lambda^{2/(N+6)}/2,\lambda)
=
\n_j^{(\alpha)}(0).
\end{align*}
  \item 
We explain why 
in the definition of $I_{\alpha,j}^\A[\vp](m,\lambda)$
we consider 
$\s(\mu(\z)m^{1/2}D_m \vp)$ 
instead of 
$\s(\mu(\z)D_m \vp)$.
In order to prove Theorem \ref{thm:main:1}, 
we use Lemma \ref{lem:V} below.
The estimate in this lemma is useful 
when $\nr{\vp}\lesssim 1$.
On the other hand, 
if we consider $\s(\mu(\z)D_m \vp)$,
then this estimate would be of the order $m^{-p}$ 
as $m\to +0$ for some $p\ge N/2$ 
because of the relation $\nr{D_m \vp}\sim m^{-1/2}$.
Consequently, 
it would be more difficult to prove 
a theorem equivalent to Theorem \ref{thm:main:1}.
  \item 
Under (H3), (H4), and (H5),
we have the reconstruction formula 
for $\nn^{(\alpha)}(0)$ for any $|\alpha|=N$ 
provided that $\nn^{(\beta)}(0)$ with $|\beta|\le N-4$ is known.
  \item 
The reason why in Theorem \ref{thm:main:1}
we need the assumption that $N$ is odd 
if $\alpha\in \Ajk{j}{k}$ with $k=2,\cdots,6$
is that we use the positivity 
of 
$d_{\vp,\alpha,m,R}$
and 
$d_{\vp,\alpha,m,M}$.
  \item 
In Remark \ref{rem:prf:thm:main:1} 
below, 
we explain 
why we set Hypotheses (H3)--(H5).
\end{enumerate}
\end{rem}

\begin{rem}\label{rem:2:thm:main:1}
Recently,  
\cite{KCDS-pre-NLD} studied the nonlinear Dirac equation in one space dimension
associated with the combined generalized massive Gross--Neveu 
and generalized massive Thirring models: 
\begin{align}\label{g-ABS}
(i\gamma^\mu\partial_\mu - m)\Psi + 
g_1 (\bar{\Psi}\Psi)^\kappa \Psi + 
g_2 \left[ (\bar{\Psi}\gamma_\nu\Psi)(\bar{\Psi}\gamma^\nu\Psi) 
\right]^{\frac{\kappa-1}{2}}(\bar{\Psi}\gamma_\mu\Psi)\gamma^\mu\Psi = 0.
\end{align} 
Here, 
$\Psi=\Psi(t,x)$ is a $\C^2$-valued unknown function on $\R\times \R$, 
$g_1$ and $g_2$ are positive coupling constants 
representing the scalar and vector self--interactions, 
respectively, 
and $\kappa$ is a positive number. 
If we set 
\begin{align*}
\Psi=\binom{u}{v}, \quad  
m=1, \quad
\gamma^0=\Miijj{0}{-1}{-1}{0},
\quad\text{and}\quad
\gamma^1=\Miijj{0}{1}{-1}{0},
\end{align*} 
then (\ref{g-ABS}) is expressed as  
\begin{align*}
i \Pt \uu - \mh \uu = \nn_{g_1,g_2}(\uu),
\quad
\uu=\binom{u}{v},
\end{align*} 
where 
\begin{align*}
\nn_{g_1,g_2}(\uu)
=
\binom{
\Pa_{\bar{u}}P(\uu)}{
\Pa_{\bar{v}}P(\uu)
},
\quad
P(\uu)
=
\frac{(-1)^\kappa      g_1 }{\kappa+1} (\bar{u}v + u\bar{v})^{\kappa+1} 
- 
\frac{   2^{\kappa +1} g_2 }{\kappa+1} |u|^{\kappa+1} |v|^{\kappa+1}.
\end{align*} 
Let us consider the case where $\kappa$ is an odd integer not smaller than 3, 
and $g_1$ and $g_2$ are unknown. 
Since $\nn_{g_1,g_2}$ satisfies (H1) and (H2), 
one can uniquely determine the exact values of $g_1$ and $g_2$ 
by Theorem \ref{thm:main:1}. 
We remark that in this case, 
one must determine two distinct coupling constants associated 
with the derivatives of the same degree of $\nn_{g_1,g_2}$. 
Therefore, 
the massless limit introduced above appears to be indispensable for this purpose.
\end{rem}

Under (H1) and (H3)--(H5),
even if $\nn^{(\beta)}(0)$ with $|\beta|\le N-4$ is unknown, 
we can still establish the reconstruction formula 
for $\nn^{(\alpha)}(0)$ for any $\alpha \in \Z^4$ with $|\alpha|=N$.

\begin{theorem}\label{thm:main:2}
Suppose (H1) and (H3)--(H5).
Choose an even function 
$\vp \in \Sc \cap B_{\eta_0}H^1$ 
so that $\supp\vp \subset [-1,1]$, 
and that $\vp>0$ on $(-1,1)$.
Fix $N\in \Z$ with $N\ge 5$.
Let $k$ and $N_0$ be non-negative integers 
such that 
$5\le N_0 <9$ and $N=4k+N_0$.
In other words, 
we put 
$k=\lfloor (N-5)/4 \rfloor$ 
and 
$N_0=N-4k$.
We also set 
\begin{align*}
p(N)=\prod_{\ell=0}^k \frac{N+6}{N_0+4\ell+6}
\quad\text{and}\quad
q(N)=\prod_{\ell=0}^k \frac{1  }{N_0+4\ell+6}.
\end{align*} 
Note that $p(N)=1$ when $5\le N<9$.
Let $\lambda$ satisfy 
\begin{align*}
0<\lambda<
\braa{\frac{1}{N+1}}^{1/p(N)}.
\end{align*} 
We define the (given) mapping $\mZ_N[\lambda]:\Z^4\to \C^2$
inductively on $N$ as follows:
\begin{itemize}
  \item 
If $N<9$, 
then for $\alpha\in \Z^4$, 
\begin{align*}
\mZ_N[\lambda](\alpha)=
\begin{cases}
\displaystyle\binom{0}{0} 
& \text{if $|\alpha|<5$,}
\\[2ex]
\displaystyle{
\binom{
K_{\alpha,1}[\vp](\lambda^{2q(|\alpha|)}/2,\lambda)
}{
K_{\alpha,2}[\vp](\lambda^{2q(|\alpha|)}/2,\lambda)
}}
& \text{if $5\le |\alpha|\le N$,}
\\[2ex]
\displaystyle\binom{0}{0} 
& \text{if $|\alpha|>N$.}
\end{cases}
\end{align*} 
  \item 
If $N\ge 9$, 
and $\mZ_5[\lambda],\mZ_6[\lambda],\cdots,\mZ_{N-1}[\lambda]$ 
have been defined, 
then for $\alpha\in \Z^4$,
\begin{align*}
\mZ_N[\lambda](\alpha)=
\begin{cases}
\mZ_{N-1}[\lambda](\alpha)
& \text{if $|\alpha|<N$,}
\\[0.5ex]
\displaystyle{
\binom{
K_{\alpha,1}^{\mZ_{N-4}[\lambda]}[\vp](\lambda^{2q(N)}/2,\lambda^{p(N)})
}{
K_{\alpha,2}^{\mZ_{N-4}[\lambda]}[\vp](\lambda^{2q(N)}/2,\lambda^{p(N)})
}}
& \text{if $|\alpha|= N$,}
\\[2ex]
\displaystyle\binom{0}{0} 
& \text{if $|\alpha|> N$.}
\end{cases}
\end{align*} 
\end{itemize}
It follows that 
if $\alpha\in \Z^4$ satisfies $|\alpha|= N$,
then for any $r\in (0,1/(N+1))$, 
we have 
\begin{align*}
\abs{
\mZ_N[r^{1/p(N)}](\alpha)
-
\nn^{(\alpha)}(0)
}
\lesssim_{N,\vp}
r^{1/(N+6)}.
\end{align*} 
In particular, 
we obtain the reconstruction formula 
\begin{align}\label{est:thm:main:2}
\lim_{r\to +0}
\mZ_N[r^{1/p(N)}](\alpha)
=
\nn^{(\alpha)}(0)
\end{align} 
even if $\nn^{(\beta)}(0)$ with $|\beta|\le N-4$ is unknown.
\end{theorem}

\subsection{Strategy}
We now give a sketch of the proof of Theorem \ref{thm:main:1}.
Set $N\ge 5$ and $\f \in B_{\eta_0}X$.
Let $\uu$ be the solution 
to $\uu=e^{-it\mh}\f+\Gs \nn(\uu)$.
We see from (\ref{id:S-op}) that 
\begin{align}\label{id:strategy:1}
i
\bra{
\s(\f)-\f,\g
}
=
\int_{\R}
\bra{
\nn(\uu)
,
e^{-it\mh}\g
}
dt.
\end{align} 
Set $n\in \Z$ with $4n+9>N$.
By the Taylor theorem,
we have as $\nr{\f}_X\to 0$,
\begin{align*}
\nn(\uu)
&=
\nn\braa{ e^{-it\mh}\f+\Gs \nn(\uu) }
\\
&=
\nn(e^{-it\mh}\f)
+
\sum_{1\le |\beta| \le n}
\frac{1}{\beta!}
\brad{ \Gs \nn(\uu) }^\beta 
\nn^{(\beta)}(e^{-it\mh}\f)
+
O\braa{ \nr{\f}_X^{5(n+1)} }
\end{align*} 
in $L^\infty X$.
We now define functions $A_m[\g]$ ($\g\in X$ and $m\in \Z$)
inductively as follows:
\begin{itemize}
  \item 
$A_0[\g]=\nn(e^{-it\mh}\g)$.
  \item 
If $m\ge 1$, 
then 
\begin{align*}
A_m[\g]
=
A_0[\g]
+
\sum_{1\le |\beta| \le m}
\frac{1}{\beta!}
\brad{ \Gs A_{m-|\beta|}[\g] }^\beta 
\nn^{(\beta)}(e^{-it\mh}\g).
\end{align*} 
\end{itemize}
Then we see from Proposition 
\ref{prop:A} below 
that 
\begin{align*}
\nn(\uu) = A_n[\f] + O\braa{ \nr{\f}_X^{4n+9} }
\quad
\text{as $\nr{\f}_X \to 0$ in $L^\infty X$.}
\end{align*} 
The method in \cite{S-2024-ISP,S-2025-ISP} 
would lead to the consideration of an identity associated with 
\begin{align*}
\frac{d^N}{d\lambda^N}A_n[\lambda \f].
\end{align*} 
However, 
we then face the difficulty of the invertibility 
of the matrix (\ref{matrix}) introduced above. 
To avoid this difficulty,
we show that for any $\vp\in H^1$, 
$n\in \Z$ and $\alpha\in \Z^4$ with $N<4n+9$ and $|\alpha|=N$, 
\begin{align*}
\Pa_{\z}^{\alpha} A_n[\mu(\z)\vp]
=
\wW_\alpha[\vp]+O(|\z|)
\quad\text{as $|\z|\to 0$},
\end{align*}  
which reduces the relation for $V_\alpha[\vp]$:
\begin{align}
&
i
\brac{
\lambda^{-|\alpha|}
\vvd_\lambda^\alpha
\bra{
\s (\mu(\z) \vp), \mu(\e_j)\vp
}
}_{\z=\binom{\lambda}{\lambda}}
-
\int_{\R}
\bra{
W_\alpha[\vp](t),
e^{-it\mh}\mu(\e_j) \vp
}
dt
\nonumber
\\
&=
\int_{\R}
\bra{
V_\alpha[\vp],
e^{-it\mh}\mu(\e_j) \vp
}
dt
+O(\lambda)
\quad\text{as $\lambda\to +0$, $j=1,2$.}
\label{identify:V}
\end{align}
In order to identify $\nn^{(\alpha)}(0)$ 
for all $\alpha\in \Z^4$ with $|\alpha|=N$, 
we assume (H3)--(H5), 
and use the massless limit.
For example,
if $\alpha \in \Z^4$ satisfies $|\alpha|=N$ 
and $\alpha \in \Ajk{1}{1}$,
then 
we have 
\begin{align*}
m^{-2-N/2}
&
\int_{\R}
\bra{
V_\alpha[m^{1/2}D_m \vp],
e^{-it\mh}\mu(\e_1) D_m \vp
}
dt
\\
&=
\int_{\R\times \R}
\brad{ \vp(x-t),\vp(x+t) }^{
\alpha+(0,1,0,0)}
d(t,x)
\,
\nn^{(\alpha)}(0)
+
O\braa{m^{1/2}}
\end{align*} 
as $m\to +0$.

As for the proof of Theorem \ref{thm:main:2},
we apply the method in \cite{S-2024-ISP,S-2025-ISP}.
More precisely, 
if $N\ge 9$, 
then instead of $W_\alpha[\vp]$, 
we use the function $W_\alpha^{\mZ_{N-4}[\lambda]}[\vp]$.
Here, 
$\mZ_{N-4}[\lambda]$ is a given function from $\Z^4$ to $\C^2$ such that 
$\mZ_{N-4}[\lambda](\delta)$ is an approximation of $\nn^{(\delta)}(0)$ 
for any $\delta\in \Z^4$ with $|\delta|\le N-4$.
\\

The rest of this paper is organized as follows.
In Section 2, 
we review some preliminaries, 
and provide an outline of the proof of Proposition \ref{prop:S}.
In Section 3,
we establish (\ref{identify:V})
by using a modified version of the method 
in \cite{S-2024-ISP,S-2025-ISP}. 
In Section 4, 
we study some estimates for functions $R_{\vp,m}$, 
$L_{\vp,m}$, 
and $M_{\vp,m}$.
In Sections 5 and 6, 
we prove Theorems \ref{thm:main:1} and \ref{thm:main:2},
respectively.

\section{Preliminaries}\label{sec:2}
In this section,
we introduce preliminaries and prove Proposition \ref{prop:S},
which is the existence theorem for $\s$.
In order to calculate $\nn$,
we often use the following two propositions:
\begin{prop}[The Taylor theorem]\label{prop:Taylor}
If $h\in C^\infty(\C^2;\C^2)$, 
then for any $n \in \Z$ and $\z,\w\in \C^2$, 
\begin{align*}
h(\z+\w)
=
\sum_{|\alpha| \le n}
\frac{\brad{ \w }^\alpha}{\alpha!} h^{(\alpha)}(\z)
+
(n+1)
\sum_{|\alpha|=n+1}\int_0^1
\frac{(1-\theta)^n}{\alpha!} 
\brad{ \w }^\alpha h^{(\alpha)}(\z+\theta\w)
d\theta.
\end{align*}  
\end{prop}

\begin{proof}
For $\lambda\in \R$,
we define $f(\lambda)=h(\z+\lambda \w)$.
Then for any $k\in \Z$, 
we have 
\begin{align*}
f^{(k)}(\lambda)
=
\sum_{|\beta|=k}\frac{k!}{\beta!}
\brad{ \w }^\beta 
h^{(\beta)}(\z+\lambda \w),
\quad
\lambda\in \R.
\end{align*} 
Using the standard Taylor theorem,
we obtain 
\begin{align*}
f(1)
=
\sum_{k=0}^n \frac{1}{k!}f^{(k)}(0)
+
\int_0^1
\frac{(1-\theta)^n}{n!} 
f^{(n+1)}(\theta)
d\theta.
\end{align*} 
Hence the desired identity holds true.
\end{proof}

\begin{prop}\label{prop:max}
Under Hypothesis (H1),  
there exists a sequence $(\M_N)_{N=0}^\infty$ 
of monotone increasing functions on $[0,\infty)$ 
such that the following properties hold:
\begin{itemize}
  \item 
For any $N\in \Z$, 
$0\le \M_N \le \M_{N+1}$.
  \item 
For any $\vv\in Z$, 
we have 
\begin{align}\label{est:1:prop:max}
\nr{ \nn(\vv)}_{L^1 X}
\lesssim 
\M_0 \braa{ \nr{\vv}_{L^\infty X}} 
\nr{\vv}_Z^5 .
\end{align} 
Furthermore, 
for any 
$N\in \N$,  
$\alpha,\beta\in \Z^4$ with 
$0<|\alpha|\le N$
and 
$|\beta|+(5-|\alpha|)_+ \ge 5$, 
and for any 
$\w_1,\cdots,\w_{|\beta|},\vv \in Z$,
we obtain  
\begin{align}\label{est:2:prop:max}
\nr{ 
\braa{
\prod_{\ell=1}^{|\beta|}
\cc_{\beta,\ell}[\w_\ell]  
}
\nn^{(\alpha)}(\vv)}_{L^1 X}
&\lesssim_\beta 
\M_{N+1}\braa{ \nr{\vv}_{L^\infty X}} 
\braa{
\prod_{\ell=1}^{|\beta|}
\nr{ \w_\ell }_Z
}
\nr{\vv}_Z^{(5-|\alpha|)_+} .
\end{align} 
\end{itemize}
\end{prop}

\begin{proof}
By the Sobolev embedding $H^1 \hookrightarrow C^0(\R) \cap L^\infty$, 
there exists a positive constant $C_E$ 
such that 
$\nr{\vp}_\infty \le C_E \nr{\vp}_{H^1}$ 
for any $\vp\in H^1$.
For $N\in \Z$, 
we set for $r>0$, 
\begin{align*}
\mathcal{M}_N(r)
=
\max_{|\alpha|\le N}
\sup\brab{
\frac{\abs{\nn^{(\alpha)}(\z)}}{|\z|^{(5-|\alpha|)_+}}
\,; \,
\z\in \C^2, \ 
0<|\z|\le C_E\, r
},
\end{align*} 
and 
$\mathcal{M}_N(0)=\liminf_{r\to +0}\mathcal{M}_N(r)$.
It follows from (H1) that 
for any $N\in \Z$, 
$\mathcal{M}_N(r)$ is indeed a finite, 
monotone increasing function on $[0,\infty)$ 
satisfying 
$0\le \M_N \le \M_{N+1}$,
and that 
\begin{align*}
\abs{\nn^{(\alpha)}(\z)}
\le 
\mathcal{M}_{|\alpha|}\braa{ C_E^{-1}|\z| }
|\z|^{(5-|\alpha|)_+},
\quad
\alpha\in \Z^4, \ \z \in \C^2.
\end{align*} 
The above embedding implies that 
if $\alpha\in \Z^4$, 
$\f\in X$, 
and $x\in \R$, 
then 
\begin{align}\label{est:prf:prop:max}
\abs{\nn^{(\alpha)}(\f(x))}
&\le 
\mathcal{M}_{|\alpha|}\braa{ C_E^{-1}|\f(x)| }
|\f(x)|^{(5-|\alpha|)_+}
\le
\mathcal{M}_{|\alpha|}\braa{ \nr{\f}_X }
|\f(x)|^{(5-|\alpha|)_+}.
\end{align} 
As for the proof of (\ref{est:1:prop:max}) and (\ref{est:2:prop:max}),
we only show (\ref{est:2:prop:max}) 
with $\alpha=(1,0,0,0)$ and $\beta=(5,0,0,0)$ 
since the other case is proved similarly.
Set $\w_k=\binom{w_{k,1}}{w_{k,2}}$ ($k=1,\cdots,5$). 
Then we have 
\begin{align*}
&
\Px\brab{
\braa{
\prod_{\ell=1}^{|\beta|}
\cc_{\beta,\ell}[\w_\ell]  
}
\nn^{(\alpha)}(\vv)
}
\\
&=
\sum_{k=1}^5 
\braa{ \Px w_{k,1} } 
\braa{\prod_{
\substack{
1\le \ell \le 5
\\
\ell\neq k
}
}w_{\ell,1}
} 
\, \nn^{(1,0,0,0)}(\vv)
+
\braa{
\prod_{k=1}^5
w_{k,1}
}
\sum_{|\gamma|=1}
\brad{ \Px \vv }^\gamma
\nn^{((1,0,0,0)+\gamma)}(\vv).
\end{align*} 
Using (\ref{est:prf:prop:max}),
we obtain 
\begin{align*}
&
\abs{
\Px\brab{
\braa{
\prod_{\ell=1}^{|\beta|}
\cc_{\beta,\ell}[\w_\ell]  
}
\nn^{(\alpha)}(\vv)
}
}
\\
&\lesssim
\sum_{k=1}^5 
\abs{ \Px w_{k,1} } 
\braa{\prod_{
\substack{
1\le \ell \le 5
\\
\ell\neq k
}
}\abs{ w_{\ell,1} }
} 
\mathcal{M}_1\braa{ \nr{\vv}_{L^\infty X} } 
\abs{\vv}^4
+
\braa{
\prod_{k=1}^5
\abs{ w_{k,1} }
}
\abs{ \Px \vv }
\mathcal{M}_2\braa{ \nr{\vv}_{L^\infty X} } 
\abs{\vv}^3.
\end{align*} 
We see from the H\"older inequality and 
the embeddings 
$H^1, B^{3/4}_{\infty}\hookrightarrow L^\infty$ that 
\begin{align*}
&
\nr{
\Px\brab{
\braa{
\prod_{\ell=1}^{|\beta|}
\cc_{\beta,\ell}[\w_\ell]  
}
\nn^{(\alpha)}(\vv)
}
}_2
\lesssim
\mathcal{M}_2\braa{ \nr{\vv}_{L^\infty X} } 
\\ 
&\quad\times
\brab{
\sum_{k=1}^5 
\nr{ \Px w_{k,1} }_2 
\braa{\prod_{
\substack{
1\le \ell \le 5
\\
\ell\neq k
}
}\nr{ w_{\ell,1} }_\infty
} 
\nr{\vv}_\infty^4
+
\braa{
\prod_{k=1}^5
\nr{ w_{k,1} }_\infty
}
\nr{ \Px \vv }_2
\nr{\vv}_\infty^3
}
\\
&\lesssim
\mathcal{M}_2\braa{ \nr{\vv}_{L^\infty X} } 
\left\{
\sum_{k=1}^5 
\nr{ \Px w_{k,1} }_2 
\braa{\prod_{
\substack{
1\le \ell \le 5
\\
\ell\neq k
}
}\nr{ w_{\ell,1} }_{H^1}
} 
\nr{\vv}_{B^{3/4}_{\infty}}^4
\right.
\\
&\quad 
\left.
\phantom{
\braa{\prod_{
\substack{
1\le \ell \le 5
\\
\ell\neq k
}
}\nr{ w_{\ell,1} }_{H^1}
}
}
+
\braa{
\prod_{k=1}^4
\nr{ w_{k,1} }_{H^1}
}
\nr{ w_{5,1} }_{B^{3/4}_{\infty}}
\nr{ \Px \vv }_2
\nr{\vv}_{B^{3/4}_{\infty}}^3
\right\}.
\end{align*} 
Hence we have 
\begin{align*}
&
\nr{
\braa{
\prod_{\ell=1}^{|\beta|}
\cc_{\beta,\ell}[\w_\ell]  
}
\nn^{(\alpha)}(\vv)
}_{L^1 X}
\\
&\lesssim
\mathcal{M}_2\braa{ \nr{\vv}_{L^\infty X} } 
\left\{
\sum_{k=1}^5 
\nr{ w_{k,1} }_{L^\infty H^1} 
\braa{\prod_{
\substack{
1\le \ell \le 5
\\
\ell\neq k
}
}\nr{ w_{\ell,1} }_{L^\infty H^1}
} 
\nr{\vv}_{L^4 B^{3/4}_{\infty}}^4
\right.
\\
&\quad 
\left.
\phantom{
\braa{\prod_{
\substack{
1\le \ell \le 5
\\
\ell\neq k
}
}\nr{ w_{\ell,1} }_\infty
}
}
+
\braa{
\prod_{k=1}^4
\nr{ w_{k,1} }_{L^\infty H^1}
}
\nr{ w_{5,1} }_{L^4 B^{3/4}_{\infty}}
\nr{ \vv }_{L^\infty X}
\nr{\vv}_{L^4 B^{3/4}_{\infty}}^3
\right\}
\\
&\lesssim
\mathcal{M}_2\braa{ \nr{\vv}_{L^\infty X} } 
\braa{
\prod_{\ell=1}^{5}
\nr{ \w_\ell }_Z
}
\nr{\vv}_Z^4.
\end{align*}
Therefore,
we obtain the desired estimate. 
\end{proof}

From the proof of Proposition \ref{prop:max}, 
we have the following corollary, 
which is used in Section \ref{sec:prf:main:2}.
\begin{cor}\label{cor:prop:max}
Let $\beta\in \Z^4$ satisfy $|\beta|\ge 5$.
It follows that for any $\w_1,\cdots,\w_{|\beta|}\in Z$,
\begin{align*}
\nr{ \prod_{\ell=1}^{|\beta|} \cc_{\beta,\ell}(\w_\ell) }_{L^1H^1}
\lesssim_{|\beta|}
\prod_{\ell=1}^{|\beta|} \nr{\w_\ell}_Z.
\end{align*} 
\end{cor}

Recall that we have defined $\omega_m=\sqrt{m^2-\Px^2}$ 
for $m>0$.
We next introduce the following two propositions 
for solutions to linear Klein-Gordon equations:
\begin{prop}\label{prop:decay}
For any $\vp\in \Sc$,
$r\in [2,\infty]$,
$\theta\in [0,1]$,
$m\in (0,1]$, 
and $t\in \R \setminus \{ 0 \}$, 
we have 
\begin{align*}
\nr{R_{\vp,m}(t)}_r,
\nr{L_{\vp,m}(t)}_r
\lesssim_\vp
|m^2 t|^{-\theta (1/2-1/r)} 
\end{align*} 
and 
\begin{align*}
\nr{M_{\vp,m}(t)}_r
\lesssim_\vp
m^\sigma
|m^2 t|^{-\theta (1/2-1/r)} .
\end{align*}
Here, 
$\sigma=1\wedge (2+\theta)(1/2-1/r)$.  
\end{prop}

\begin{proof}
We now use the operator 
$D_m$ defined in Section 1 
by $(D_m \vp)(x)=\vp(x/m)$ for $x\in \R$.
Since 
$R_{\vp,m}(t)=D_m^{-1}R_{D_m \vp}(mt)$,
by the standard time decay estimate 
for the free Klein-Gordon equation (see, e.g., \cite{B-1985-decay}), 
we have 
\begin{align*}
\nr{R_{\vp,m}(t)}_r
&=
\nr{D_m^{-1}R_{D_m \vp}(mt)}_r
=
m^{-1/r} \nr{R_{D_m \vp}(mt)}_r
\\
&\lesssim
m^{-1/r} |mt|^{-\theta(1/2-1/r)} 
\nr{D_m \vp}_{B_{r^\prime}^{(2+\theta)(1/2-1/r)}}
\\
&\lesssim
m^{-1/r} |mt|^{-\theta(1/2-1/r)} 
m^{-(2+\theta)(1/2-1/r)}
m^{1/r^\prime}
\nr{\vp}_{B_{r^\prime}^{(2+\theta)(1/2-1/r)}}
\\
&\lesssim_\vp
  m^{-2\theta(1/2-1/r)}
|t|^{- \theta(1/2-1/r)}.
\end{align*}
Similarly,
we also obtain the estimate for $L_{\vp,m}$.
As for $M_{\vp,m}$, 
we see from the identity $M_{\vp,m}(t)=D_m^{-1}M_{D_m\vp}(mt)$
that 
\begin{align*}
\nr{M_{\vp,m}(t)}_r
&=
m^{-1/r} \nr{M_{D_m \vp}(mt) }_r
\\
&\lesssim
m^{-1/r} |mt|^{-\theta(1/2-1/r)} 
\nr{D_m \vp}_{B_{r^\prime}^{(2+\theta)(1/2-1/r)-1}}
\\
&\lesssim
m^{-1/r} |mt|^{-\theta(1/2-1/r)} 
m^{-((2+\theta)(1/2-1/r)-1)_+}
m^{1/r^\prime}
\nr{\vp}_{B_{r^\prime}^{(2+\theta)(1/2-1/r)-1}}
\\
&\lesssim_\vp
  m^{- \theta(1/2-1/r)+2(1/2-1/r)-((2+\theta)(1/2-1/r)-1)_+}
|t|^{- \theta(1/2-1/r)}
\\
&=
  m^{(2+\theta)(1/2-1/r)-((2+\theta)(1/2-1/r)-1)_+}
|m^2 t|^{- \theta(1/2-1/r)}.
\end{align*} 
Since 
$(2+\theta)(1/2-1/r)-((2+\theta)(1/2-1/r)-1)_+=\sigma$,
we have the desired estimate for $M_{\vp,m}$.
\end{proof}

\begin{prop}\label{prop:St}
\begin{enumerate}[(1)]
  \item 
The operator $\vp \mapsto e^{it\omega}\vp$ is bounded 
from $H^1$ to $L^1H^1 \cap L^4 B^{3/4}_\infty$.
  \item 
The operator $\Gs$ is bounded from $L^1X$ to $Z$.
  \item 
For any $\vp\in \Sc$ and $m\in (0,1]$,
we have 
\begin{align*}
\nr{R_{\vp,m}}_{L^{20/3}L^5},\nr{L_{\vp,m}}_{L^{20/3}L^5} 
\lesssim_\vp m^{-3/10}
\end{align*} 
and 
\begin{align*}
\nr{M_{\vp,m}}_{L^5L^{10}}
\lesssim_\vp m^{3/5}.
\end{align*} 
\end{enumerate}
\end{prop}

\begin{proof}
We see from 
the Strichartz estimates 
for the linear Klein-Gordon equation 
(see, e.g., \cite{W-1999-St}) 
that (1) and (2) hold true, 
and that 
\begin{align}
\nr{ e^{it\omega}\vp }_{L^{20/3}L^5}
&\lesssim
\nr{\vp}_{H^{9/20}},
\quad
\vp \in H^{9/20},
\label{est:1:prf:prop:St}
\\
\nr{ e^{it\omega}\vp }_{L^{5}L^{10}}
&\lesssim
\nr{\vp}_{H^{3/5}},
\quad
\vp \in H^{3/5}.
\label{est:2:prf:prop:St}
\end{align}

It suffices to prove (3).
Since 
$R_{\vp,m}(t)=D_m^{-1}R_{D_m \vp}(mt)$,
we see from (\ref{est:1:prf:prop:St}) that 
\begin{align*}
\nr{R_{\vp,m}}_{L^{20/3}L^5}
&=
\nr{D_m^{-1}R_{D_m \vp}(mt)}_{L^{20/3}L^5}
=
m^{-3/20} m^{-1/5} 
\nr{R_{D_m \vp}(t)}_{L^{20/3}L^5}
\\
&\lesssim
m^{-7/20} \nr{D_m \vp}_{H^{9/20}}
\lesssim
m^{-7/20} m^{1/2-9/20} \nr{\vp}_{H^{9/20}}
\lesssim_\vp
m^{-3/10}. 
\end{align*} 
Similarly,
we also obtain the estimate for $L_{\vp,m}$.
As for $M_{\vp,m}$, 
by the identity $M_{\vp,m}(t)=D_m^{-1}M_{D_m\vp}(mt)$
and (\ref{est:2:prf:prop:St}), 
we have 
\begin{align*}
\nr{M_{\vp,m}}_{L^5L^{10}}
&=
\nr{D_m^{-1}M_{D_m \vp}(mt)}_{L^5L^{10}}
=
m^{-1/5} m^{-1/10} 
\nr{M_{D_m \vp}(t)}_{L^5L^{10}}
\\
&\lesssim
m^{-3/10} \nr{D_m \vp}_{H^{3/5-1}}.
\end{align*} 
By the Sobolev embedding 
$L^{10/9} \hookrightarrow H^{-2/5}$,
we have 
\begin{align*}
\nr{D_m \vp}_{H^{-2/5}}
\lesssim
\nr{D_m \vp}_{L^{10/9}}
\lesssim
m^{9/10} \nr{\vp}_{L^{10/9}},
\end{align*} 
which completes the proof of the desired estimate for $M_{\vp,m}$.
\end{proof}

We next show that 
the difference operators $\vvd_\lambda^\alpha$
are nearly equal to 
the Wirtinger derivatives $\Pa_\z^\alpha$ 
in some sense.
\begin{prop}\label{prop:dl}
Fix $N\in \Z$ and $\lambda_0>0$.
Assume $h\in C^{N+1}(\C^2;\C)$ and that 
\begin{align*}
c_{N,\lambda_0}
:=
\sup_{
\substack{
0<|\z|<\lambda_0
\\
|\alpha|=N+1
}
}
\abs{
h^{(\alpha)}(\z)
}
<\infty.
\end{align*} 
Then we have 
\begin{align}\label{id:prop:dl}
\max_{|\alpha|=N}
\abs{
\lambda^{-|\alpha|}
\braa{ \vvd_\lambda^\alpha h }(\lambda,\lambda)
-
h^{(\alpha)}(\lambda,\lambda)
}
\lesssim_{N} c_{N,\lambda_0}\lambda,
\quad
\lambda\in \braa{
0,\frac{\lambda_0}{N+1}
}.
\end{align} 
\end{prop}

\begin{proof}
Fix $\lambda>0$ and $\alpha\in \Z^4$ with $|\alpha|=N$.
Then we have 
\begin{align*}
\brad{
x_1-ix_2-1+i,x_3-ix_4-1+i
}^\alpha
=
2^{|\alpha|}
\sum_{\substack{
\beta \in \Z^4
\\
|\beta|\le N
}}
C_{\alpha,\beta}
(x_1,x_2,x_3,x_4)^\beta,
\quad
x_1,\cdots,x_4\in \R
\end{align*} 
for some complex constants $C_{\alpha,\beta}$.
Hence we
obtain 
\begin{align}\label{id:1:prf:prop:dl}
\vvd_\lambda^\alpha=
\sum_{\substack{
\beta \in \Z^4
\\
|\beta|\le N
}}
C_{\alpha,\beta}\,
\tau[
\beta_1 \lambda+i\beta_2 \lambda,
\beta_3 \lambda+i\beta_4 \lambda
].
\end{align} 
It follows from direct calculation that 
\begin{align*}
\vvd_\lambda^\alpha \brad{\z}^\gamma
=
\left\{
  \begin{array}{cl}
\alpha! \lambda^N&(\gamma=\alpha),\\
0                &(|\gamma|\le N \text{ and } \gamma\neq \alpha)\\
  \end{array}
\right.,
\quad
\z \in \C^2,  
\end{align*} 
which implies that 
\begin{align}\label{id:2:prf:prop:dl}
\sum_{\substack{
\beta \in \Z^4
\\
|\beta|\le N
}}
C_{\alpha,\beta}\,
\brad{
\beta_1 +i\beta_2 ,
\beta_3 +i\beta_4 
}^\gamma
=
\left\{
  \begin{array}{cl}
\alpha!          &(\gamma=\alpha),\\
0                &(|\gamma|\le N \text{ and } \gamma\neq \alpha).\\
  \end{array}
\right.
\end{align}
We see from Proposition \ref{prop:Taylor} that 
for any $\beta \in \Z^4$ with $|\beta|\le N$, 
\begin{align*}
&
h(
\lambda+\beta_1 \lambda+i\beta_2 \lambda,
\lambda+\beta_3 \lambda+i\beta_4 \lambda
)
=
\sum_{\substack{
\gamma \in \Z^4
\\
|\gamma|\le N
}}
\frac{1}{\gamma!}
\brad{
\beta_1 +i\beta_2 ,
\beta_3 +i\beta_4
}^\gamma 
h^{(\gamma)}(\lambda,\lambda)
\lambda^{|\gamma|}
\\
&\quad +
(N+1)
\sum_{\substack{
\gamma \in \Z^4
\\
|\gamma|= N+1
}}
\int_0^1
\frac{(1-\theta)^N}{\gamma!}
\brad{
\beta_1 +i\beta_2 ,
\beta_3 +i\beta_4
}^\gamma 
h^{(\gamma)}(
\lambda_{\beta_1,\beta_2,\theta},
\lambda_{\beta_3,\beta_4,\theta}
)
\lambda^{N+1}
d\theta.
\end{align*}
Here, 
we have defined 
$
\lambda_{a,b,\theta}
=
\lambda+\theta(a\lambda+ib\lambda),
$
($a,b \in \R$). 
By (\ref{id:1:prf:prop:dl}) and (\ref{id:2:prf:prop:dl}),
we have 
\begin{align*}
&
\lambda^{-N}
\braa{ \vvd_\lambda^\alpha h }(\lambda,\lambda)
=
h^{(\alpha)}(\lambda,\lambda)
\\
&\quad +
(N+1)\lambda
\sum_{\substack{
\gamma \in \Z^4
\\
|\gamma|= N+1
}}
\int_0^1
\frac{(1-\theta)^N}{\gamma!}
\sum_{\substack{
\beta \in \Z^4
\\
|\beta|\le N
}}
C_{\alpha,\beta}
\brad{
\beta_1 +i\beta_2 ,
\beta_3 +i\beta_4
}^\gamma 
h^{(\gamma)}(
\lambda_{\beta_1,\beta_2,\theta},
\lambda_{\beta_3,\beta_4,\theta}
)
d\theta.
\end{align*} 
Therefore,
(\ref{id:prop:dl}) holds true.
\end{proof}

For $\f \in X$ and $t\in \R$, 
we define $(U\f)(t)=e^{-it\mh}\f$.
Then Equation (\ref{ieq:-}) is expressed by 
$\uu=U\f_- + \Gs \nn(\uu)$. 
We now briefly prove the existence of $\s$.

\begin{proof}[Outline of the proof of Proposition \ref{prop:S}]
Assume (H1).
Put $r,\eta>0$, 
$\f_- \in B_\eta X$
and $\vv,\widetilde{\vv}\in B_r Z$.
It follows from 
Propositions \ref{prop:Taylor} and \ref{prop:max}
that 
\begin{align*}
\nr{
\nn(\vv)-\nn\braa{\widetilde{\vv}}
}_{L^1 X}
\lesssim
\M_2(r) 
\braa{
\nr{\vv}_Z \vee \nr{\widetilde{\vv}}_Z
}^4
\nr{
\vv-\widetilde{\vv}
}_Z.
\end{align*} 
We see from Proposition \ref{prop:St} that 
the mapping 
\begin{align*}
\Phi:B_r Z \ni \vv \mapsto 
U\f_- + \Gs \nn(\vv) \in Z 
\end{align*} 
is well-defined, 
and that  
\begin{align*}
\nr{\Phi(\vv)}_Z
&\le 
C\nr{\f_-}_X+C\M_2(r)\nr{\vv}_Z^5,
\\
\nr{\Phi(\vv)-\Phi\braa{\widetilde{\vv}}}_Z
&\le
C\M_2(r)
\braa{
\nr{\vv}_Z \vee \nr{\widetilde{\vv}}_Z
}^4
\nr{\vv-\widetilde{\vv}}_Z.
\end{align*}
Here, 
$C$ is a positive constant independent 
of $\vv,\widetilde{\vv}$ and $r$.
Choose $r$ and $\eta$ so that 
\begin{align*}
r=2C\nr{\f_-}_X
\quad\text{and}\quad
32 C^5 \M_2(r) \nr{\f_-}_X^4 \le 1.
\end{align*}
Then $\Phi$ is a contraction on the complete metric space $B_r Z$.
Furthermore, 
the unique fixed point $\uu$ satisfies 
$\nr{\nn(\uu)}_{L^1 X} \lesssim \nr{\f_-}_X^5$.
It follows from the standard argument that 
$\uu$ is the unique solution to the equation 
$\uu = \Phi (\uu)$.
The desired properties of $\s$ follow immediately.
\end{proof}

\section{A formula for $V_\alpha[\vp]$}
In this section, 
we give a relation for functions 
\begin{align*}
V_\alpha[\vp]
:=
\sum_{\gamma\le \alpha}
\frac{\alpha!}{\gamma! (\alpha-\gamma)!}
\brad{R_\vp,L_\vp}^{\alpha-\gamma}
\brad{iM_\vp,iM_\vp}^{\gamma}
\nn^{(\alpha-\gamma+\widetilde{\gamma})}(0)
\end{align*} 
by a modified version of the method 
in \cite{S-2024-ISP,S-2025-ISP}. 
As a first step,
we expand the nonlinearity $\nn$
by using functions $A_n[\f]$ 
defined in Subsection 1.3.

\begin{prop}\label{prop:A}
Assume (H1).
Put $\f\in B_{\eta_0} X$. 
Let $\uu\in Z$ be the solution to 
the equation
\begin{align}\label{ieq}
\uu=U\f + \Gs \nn (\uu).
\end{align} 
It follows that 
for any $n\in \Z$,
\begin{align}\label{est:prop:A}
\nr{
A_n[\f]-\n(\uu)
}_{L^1X} 
\lesssim_n
\nr{\f}_X^{4n+9}.
\end{align} 
\end{prop}

\begin{rem}\label{rem:prop:A}
Set $\f \in B_{\eta_0}X$ and $n\in \Z$.
\begin{enumerate}[(1)]
  \item 
We see from (\ref{id:S-op}) that 
\begin{align*}
\nr{
i\braa{ 
\s(\f)-\f
}
-
\int_{\R}e^{it\mh}A_n[\f] dt
}
\lesssim_n \nr{\f}_X^{4n+9}.
\end{align*} 
  \item 
Recall the inequality 
$\nr{\nn(\uu)}_{L^1X}\lesssim \nr{\f}_X^5$
in Proposition \ref{prop:S}.
By (\ref{est:prop:A}),
we have $A_n[\f]\in L^1X$ and 
$\nr{A_n[\f]}_{L^1X}\lesssim_n \nr{\f}_X^5$.
\end{enumerate}
\end{rem}

\begin{proof}[Proof of Proposition \ref{prop:A}]
Set $n\in \Z$.
By Proposition \ref{prop:Taylor},
we have 
\begin{align*}
&
\nr{
A_n[\f]-\nn(\uu)
}_{L^1X}
=
\nr{
A_n[\f]-\nn\braa{U\f+\Gs\nn(\uu)}
}_{L^1X}
\\
&\lesssim_n
\sum_{|\beta|= n+1}
\int_0^1 
\nr{
\brad{ \Gs\nn(\uu) }^{\beta} 
\nn^{(\beta)}\braa{ U\f + \theta \Gs\nn(\uu) }
}_{L^1 X}
d\theta
\\
&\quad 
+
\sum_{1\le |\beta| \le n}
\nr{
\braa{
\brad{ \Gs A_{n-|\beta|}[\f] }^\beta -
\brad{ \Gs \nn(\uu) }^\beta
}
\nn^{(\beta)}(U\f)
}_{L^1X}
\\
&=
I_n+II_n.
\end{align*}  
Here,
for $n=0$ we regard $II_n$ as $0$.
Using Propositions  
\ref{prop:max} and \ref{prop:St},
we obtain 
\begin{align*}
I_n
&\lesssim_n
\M_{n+1}\braa{
\nr{U\f}_{L^\infty X}+\nr{\Gs\nn(\uu)}_{L^\infty X}
} \nr{\Gs\nn(\uu)}_Z^{n+1}
\braa{ \nr{U\f}_Z+ \nr{\Gs\nn(\uu)}_Z  }^{(5-n-1)_+}
\\
&\lesssim_n
\M_{n+1}\braa{ \nr{\f}_X+\nr{\nn(\uu)}_{L^1 X}  }
\nr{\nn(\uu)}_{L^1X}^{n+1}
\braa{ \nr{\f}_X+\nr{\nn(\uu)}_{L^1 X}  }^{(4-n)_+}.
\end{align*} 
It follows from Remark \ref{rem:prop:A}(2) that 
\begin{align}\label{est:prf:prop:A}
I_n
\lesssim_n
\nr{\f}_X^{5(n+1)+(4-n)_+}
\lesssim_n 
\nr{\f}_X^{4n+9}.
\end{align} 
In particular, 
we have (\ref{est:prop:A}) when $n=0$.

We assume that there exists some $m\in \N$ 
such that (\ref{est:prop:A}) holds true 
for any $n=0,\cdots,m-1$.
Then as in the proof of (\ref{est:prf:prop:A}), 
we have
\begin{align*}
II_m
&\lesssim_m 
\sum_{1\le |\beta| \le m}
\nr{
A_{m-|\beta|}[\f] - \nn(\uu)
}_{L^1X} 
\braa{
\nr{
A_{m-|\beta|}[\f]
}_{L^1X} 
\vee
\nr{
\nn(\uu)
}_{L^1X} 
}^{|\beta|-1}
\nr{\f}_X^{(5-|\beta|)_+}
\\
&\lesssim_m
\sum_{1\le |\beta| \le m}
\nr{\f}_X^{4(m-|\beta|)+9+5(|\beta|-1)+(5-|\beta|)_+}
\lesssim_m
\nr{\f}_X^{4m+9},
\end{align*} 
which implies that 
(\ref{est:prop:A}) holds when $n=m$.
By mathematical induction,
we complete the proof.
\end{proof}

We next show that 
the derivatives 
$\Pa_\z^\alpha A_n[\mu(\z)\vp]$ 
tend to  
$\wW_\alpha[\vp]$ 
as $\z \to 0$ 
under some conditions.

\begin{prop}\label{prop:W}
Assume (H1).
Fix $\vp\in B_1H^1$.
Then for any $n\in \Z$,
the following properties hold true:
\begin{enumerate}[(D${}_n$-1)]
  \item 
$A_n[\mu(\z)\vp]\in C^\infty(\C^2;L^1X)$.
  \item 
If $\z\in B_1\C^2$ and $\alpha\in \Z^4$ then 
$
\nr{ \Pa_\z^\alpha A_n[\mu(\z)\vp] }_{L^1 X}
\lesssim_{n,|\alpha|}
|\z|^{(5-|\alpha|)_+}.
$
  \item 
If $\z\in B_1\C^2$ and $\alpha\in \Z^4$ 
with $5\le |\alpha|<4n+9$ then  
\begin{align*}
\nr{ \Pa_\z^\alpha A_n[\mu(\z)\vp]-\wW_\alpha[\vp] }_{L^1 X}
\lesssim_{n}
|\z|.
\end{align*} 
\end{enumerate}
\end{prop}

\begin{rem}\label{rem:prop:W}
Set $\vp\in B_1 H^1$, 
$n\in \Z$ and $\alpha\in \Z^4$ with
$|\alpha|<4n+9$.
It follows that 
\begin{align*}
\nr{
\Pa_\z^\alpha
\int_{\R}e^{it\mh}  A_n[\mu(\z)\vp] dt
-
\int_{\R}e^{it\mh}\wW_\alpha[\vp] dt
}
\lesssim_{n}
|\z|,
\quad
\z\in B_1\C^2.
\end{align*}  
\end{rem}

\begin{proof}[Outline of the proof of Proposition \ref{prop:W}]
We divide the proof into two steps.
\\
(Step 1)
In this proof,
we repeatedly use 
Propositions \ref{prop:max} and \ref{prop:St}.
Put $\z\in B_1\C^2$ and $\alpha\in \Z^4$.
By (\ref{id:U(t):2}),
we have for any $\beta\in \Z^4$,
\begin{align}\label{id:prf:prop:W}
\Pa_\z^\alpha \nn^{(\beta)}(U\mu(\z)\vp)
=
\sum_{\gamma\le \alpha}
\frac{\alpha!}{\gamma! (\alpha-\gamma)!}
\brad{R_\vp,L_\vp}^{\alpha-\gamma}
\brad{iM_\vp,iM_\vp}^{\gamma}
\nn^{(\beta+\alpha-\gamma+\widetilde{\gamma})}(U\mu(\z)\vp).
\end{align}
Since $|\alpha|+(5-|\alpha|)_+\ge 5$, 
we can use Proposition \ref{prop:max} 
to show 
\begin{align*}
\nr{
\Pa_\z^{\alpha} \nn(U\mu(\z)\vp)
}_{L^1X}
&\lesssim_{|\alpha|}
\braa{
\nr{R_\vp}_Z \vee \nr{L_\vp}_Z \vee \nr{M_\vp}_Z
}^{|\alpha|}
\nr{U\mu(\z)\vp}_Z^{(5-|\alpha|)_+}
\\
&\lesssim_{|\alpha|}
\nr{\vp}^{|\alpha|} \nr{\mu(\z)\vp}^{(5-|\alpha|)_+}
\lesssim_{|\alpha|} |\z|^{(5-|\alpha|)_+}.
\end{align*}
Hence (D$_1$-1) and (D$_1$-2) hold true.
Furthermore, 
for any $n\in \N$, 
we obtain (D$_n$-1) and 
\begin{align*}
&
\Pa_\z^\alpha A_n[\mu(\z)\vp]
=
\Pa_\z^{\alpha} \nn(U\mu(\z)\vp)
\\
&+
\sum_{1\le |\beta| \le n}
\frac{\alpha!}{\beta!}
\sssum
\braa{
\prod_{\ell=1}^{|\beta|}
\frac{
\cc_{\beta,\ell} \braa{ \Gs \Pa_\z^{\alpha(\ell)} A_{n-|\beta|}[\mu(\z)\vp] }
}{\alpha(\ell)!}
}
\frac{
\Pa_\z^{\alpha(0)} \nn^{(\beta)}(U\mu(\z)\vp)
}{\alpha(0)!},
\end{align*}
which implies that 
\begin{align*}
&
\nr{
\Pa_\z^\alpha A_n[\mu(\z)\vp]
}_{L^1X}
\lesssim_{n,|\alpha|}
|\z|^{(5-|\alpha|)_+}
\\
&\quad +
\sum_{1\le |\beta| \le n}
\sssum
\braa{
\prod_{\ell=1}^{|\beta|}
\nr{
\Gs \Pa_\z^{\alpha(\ell)} A_{n-|\beta|}[\mu(\z)\vp]
}_Z 
}
\nr{\vp}^{|\alpha(0)|}
\nr{\mu(\z)\vp}_X^{(5-|\alpha(0)+\beta|)_+}
\\
&\lesssim_{n,|\alpha|}
|\z|^{(5-|\alpha|)_+}
\\
&\quad +
\sum_{1\le |\beta| \le n}
\sssum
\braa{
\prod_{\ell=1}^{|\beta|}
\nr{
\Pa_\z^{\alpha(\ell)} A_{n-|\beta|}[\mu(\z)\vp]
}_{L^1 X} 
}
|\z|^{(5-|\alpha(0)+\beta|)_+}.
\end{align*} 
If (D$_m$-2) is true for any $m=0,1,\cdots,n-1$,
then we have 
\begin{align*}
&
\nr{
\Pa_\z^\alpha A_n[\mu(\z)\vp]
}_{L^1X}
\\
&
\lesssim_{n,|\alpha|}
|\z|^{(5-|\alpha|)_+}
+
\sum_{1\le |\beta| \le n}
\sssum
\braa{
\prod_{\ell=1}^{|\beta|}
|\z|^{(5-|\alpha(\ell)|)_+} 
}
|\z|^{(5-|\alpha(0)+\beta|)_+}
\\
&
\lesssim_{n,|\alpha|}
|\z|^{(5-|\alpha|)_+}.
\end{align*} 
By mathematical induction,
we see that (D$_n$-2) holds true for any $n\in \Z$.
\\
(Step 2)
Put $\z\in B_1\C^2$.
By (\ref{id:prf:prop:W}),
we have for any $\alpha\in\Z^4$, 
\begin{align}\label{id:2:prf:prop:W}
&
\nr{
\Pa_\z^\alpha \nn(U\mu(\z)\vp)
- V_\alpha[\vp]
}_{L^1X}
\nonumber
\\
&
\lesssim_{|\alpha|}
\sum_{\gamma\le \alpha}
\nr{
\brad{R_\vp,L_\vp}^{\alpha-\gamma}
\brad{iM_\vp,iM_\vp}^{\gamma}
\braa{
\nn^{(\alpha-\gamma+\widetilde{\gamma})}(U\mu(\z)\vp)
-
\nn^{(\alpha-\gamma+\widetilde{\gamma})}(0          )
}
}_{L^1X}
\nonumber
\\
&
\lesssim_{|\alpha|}
\braa{
\nr{R_\vp}_Z \vee \nr{L_\vp}_Z \vee \nr{M_\vp}_Z
}^{|\alpha|}
\nr{
U\mu(\z)\vp
}_Z^{1+(4-|\alpha|)_+}
\nonumber
\\
&
\lesssim_{|\alpha|}
\nr{\vp}^{|\alpha|+1+(4-|\alpha|)_+}|\z|^{1+(4-|\alpha|)_+}
\lesssim_{|\alpha|}
|\z|,
\end{align} 
which implies that (D$_1$-3) holds true.

Henceforth,
we assume that $n\in \N$ and $\alpha\in \Z^4$ 
satisfy $5\le |\alpha| <4n+9$.
Then we have $n\ge k_\alpha$.
Furthermore, 
if $\beta\in\Z^4$ satisfies $n>|\beta|$, 
then all of 
$\alpha(0),\alpha(1),\cdots,\alpha(|\beta|) \in \Z^4$ 
with 
$\alpha(0)+\alpha(1)+\cdots +\alpha(|\beta|)=\alpha$ 
do not satisfy 
both 
$|\alpha(0)|\ge (5-|\beta|)_+$
and 
$|\alpha(1)|,\cdots,|\alpha(|\beta|)| \ge 5$.
It follows that 
\begin{align}\label{id:3:prf:prop:W}
\nr{ \Pa_\z^\alpha A_n[\mu(\z)\vp]-\wW_\alpha[\vp] }_{L^1 X}
\lesssim_{n}
\nr{
\Pa_\z^\alpha \nn(U\mu(\z)\vp)
- V_\alpha[\vp]
}_{L^1X}
+
I^1_{n,\alpha} + I^2_{n,\alpha} + II_{n,\alpha}.
\end{align} 
Here, 
we have defined 
\begin{align*}
&
I^j_{n,\alpha}
=
\sum_{1\le |\beta| \le k_\alpha}\ssum
 I^j_{n,\beta,\alpha(0),\alpha(1),\cdots,\alpha(|\beta|)}(\z)
 \quad
 (j=1,2),
\\
&
II_{n,\alpha}
=
\sum_{1\le |\beta| \le n}
\ssssum
II_{n,\beta,\alpha(0),\alpha(1),\cdots,\alpha(|\beta|)}(\z),
\\
&
I^1_{n,\beta,\alpha(0),\alpha(1),\cdots,\alpha(|\beta|)}(\z)
\\
&\quad =
\nr{
\brab{
\braa{
\prod_{\ell=1}^{|\beta|}
\cc_{\beta,\ell} \braa{ \Gs \Pa_\z^{\alpha(\ell)} A_{n-|\beta|}[\mu(\z)\vp] }
}
-
\braa{
\prod_{\ell=1}^{|\beta|}
\cc_{\beta,\ell} \braa{ \Gs \wW_{\alpha(\ell)}[\vp] }
}
}
V_{\alpha(0),\beta}[\vp]
}_{L^1X},
\\
&
I^2_{n,\beta,\alpha(0),\alpha(1),\cdots,\alpha(|\beta|)}(\z)
\\
&=
\nr{
\braa{
\prod_{\ell=1}^{|\beta|}
\cc_{\beta,\ell} \braa{ \Gs \Pa_\z^{\alpha(\ell)} A_{n-|\beta|}[\mu(\z)\vp] }
}
\braa{
\Pa_\z^{\alpha(0)} \nn^{(\beta)}(U\mu(\z)\vp)
-
V_{\alpha(0),\beta}[\vp]
}
}_{L^1X}
\end{align*} 
and 
\begin{align*}
II_{n,\beta,\alpha(0),\alpha(1),\cdots,\alpha(|\beta|)}(\z)
=
\nr{
\braa{
\prod_{\ell=1}^{|\beta|}
\cc_{\beta,\ell} \braa{ \Gs \Pa_\z^{\alpha(\ell)} A_{n-|\beta|}[\mu(\z)\vp] }
}
\Pa_\z^{\alpha(0)} \nn^{(\beta)}(U\mu(\z)\vp)
}_{L^1X}.
\end{align*} 
By Step 1, 
(\ref{id:prf:prop:W}) and the proof of (\ref{id:2:prf:prop:W}), 
we obtain 
\begin{align*}
I^2_{n,\alpha}
\lesssim_{n}
\sum_{1\le |\beta| \le k_\alpha}\ssum
\braa{
\prod_{\ell=1}^{|\beta|}
|\z|^{(5-|\alpha(\ell)|)_+}
}
|\z|^{1+(4-|\alpha(0)+\beta|)_+}
\lesssim_{n}
|\z|
\end{align*} 
and 
\begin{align*}
&
II_{n,\alpha}
\\
&\lesssim_{n}
\sum_{1\le |\beta| \le k_\alpha}\ssssum
\braa{
\prod_{\ell=1}^{|\beta|}
|\z|^{(5-|\alpha(\ell)|)_+}
}
|\z|^{(5-|\alpha(0)+\beta|)_+}
\\
&\lesssim_{n}
|\z|.
\end{align*} 
Similarly,
we have 
\begin{align*}
&
I^1_{n,\alpha}
\\
&\lesssim_{n}
\sum_{1\le |\beta| \le k_\alpha}\ssum
\sum_{k=1}^{|\beta|}
\nr{
\Pa_\z^{\alpha(k)} A_{n-|\beta|}[\mu(\z)\vp] 
-
\wW_{\alpha(k)}[\vp]
}_{L^1X}
\prod_{
\substack{
1\le \ell \le |\beta|
\\
\ell\neq k
}
}^{|\beta|}
|\z|^{(5-|\alpha(\ell)|)_+}.
\end{align*} 
It follows from 
(\ref{id:2:prf:prop:W}) and 
(\ref{id:3:prf:prop:W}) 
that
\begin{align}\label{id:4:prf:prop:W}
&
\nr{ \Pa_\z^\alpha A_n[\mu(\z)\vp]-\wW_\alpha[\vp] }_{L^1 X}
\nonumber
\\
&\lesssim_{n}
|\z|
+
\sum_{1\le |\beta| \le k_\alpha}\ssum
\sum_{k=1}^{|\beta|}
\nr{
\Pa_\z^{\alpha(k)} A_{n-|\beta|}[\mu(\z)\vp] 
-
\wW_{\alpha(k)}[\vp]
}_{L^1X}.
\end{align}
Let $\alpha(k)$ and $\beta$ be multi-indices 
appearing in (\ref{id:4:prf:prop:W}).
Then we see that 
\begin{align*}
5
&\le |\alpha(k)| 
= |\alpha|-|\alpha(0)|-
\sum_{
\substack{
1\le \ell \le |\beta|
\\
\ell\neq k
}
}
|\alpha(\ell)|
<
4n+9
-(5-|\beta|)_+
- 5(|\beta|-1)
\\
&\le 
4(n-|\beta|)+9.
\end{align*}
Therefore, 
if (D$_m$-3) is true for any $m=0,1,\cdots,n-1$, 
then we have 
\begin{align*}
\nr{ \Pa_\z^\alpha A_n[\mu(\z)\vp]-\wW_\alpha[\vp] }_{L^1 X}
\lesssim_{n}
|\z|.
\end{align*} 
By mathematical induction, 
we see that (D$_n$-3) holds true for any $n\in \Z$.
\end{proof}

We next establish a key relation for $V_\alpha[\vp]$ 
by using $I_{\alpha,j}[\vp]$ defined in Section 1.

\begin{lem}\label{lem:V}
Assume (H1).
Set 
$\vp\in B_{\eta_0}H^1$, 
$\alpha\in \Z^4$ with $|\alpha|\ge 5$, 
$j=1,2$ 
and $\lambda\in (0,1/(|\alpha|+1))$.
It follows that 
\begin{align*}
\abs{
I_{\alpha,j}[\vp](1,\lambda)
-
\int_{\R}  \bra{V_\alpha[\vp],e^{-it\mh}\mu(\e_j)\vp}dt
}
\lesssim_{|\alpha|}
\lambda.
\end{align*}
In particular, 
the unknown value 
$\int_{\R}  \bra{V_\alpha[\vp],e^{-it\mh}\mu(\e_j)\vp}dt$
is uniquely determined by
\begin{align*}
\int_{\R}  \bra{V_\alpha[\vp],e^{-it\mh}\mu(\e_j)\vp}dt
=
\lim_{\lambda\to +0}
I_{\alpha,j}[\vp](1,\lambda)
.
\end{align*}
\end{lem}

\begin{proof}
We also set $n=k_\alpha$ and $\z\in B_{1/(|\alpha|+1)}\C^2$.
We see from Remark \ref{rem:prop:A}(2) that 
for any $\f\in B_{\eta_0}X$ and $\g \in X$,
\begin{align*}
\abs{
i\bra{ 
\s(\f)-\f,
\g
}
-
\int_{\R}  \bra{A_n[\f],e^{-it\mh}\g }dt
}
&\le 
\nr{
i\braa{ 
\s(\f)-\f
}
-
\int_{\R}e^{it\mh}A_n[\f] dt
}_X
\nr{\g}_X
\\
&\lesssim_{|\alpha|} \nr{\f}_X^{4n+9} \nr{\g}_X,
\end{align*} 
which implies that 
\begin{align*}
\abs{
i\bra{ 
\s(\mu(\z)\vp)-\mu(\z)\vp,
\mu(\e_j)\vp
}
-
\int_{\R}  \bra{A_n[\mu(\z)\vp],e^{-it\mh}\mu(\e_j)\vp}dt
}
\lesssim_{|\alpha|}
|\z|^{4n+9}.
\end{align*}
Since
$
\lambda^{-|\alpha|}
\vvd_\lambda^\alpha 
\bra{ 
\mu(\z)\vp,
\mu(\e_j)\vp
}
=0,
$
we have 
\begin{align*}
&
\abs{
\brac{
\lambda^{-|\alpha|}
\vvd_\lambda^\alpha 
\brab{
\bra{ 
i\s(\mu(\z)\vp),
\mu(\e_j)\vp
}
-
\int_{\R}  \bra{A_n[\mu(\z)\vp],e^{-it\mh}\mu(\e_j)\vp}dt
}
}_{\z=\binom{\lambda}{\lambda}}
}
\\
&\lesssim_{|\alpha|}
\lambda^{4n+9-|\alpha|}.
\end{align*} 
By Proposition \ref{prop:dl},
we obtain 
\begin{align*}
&
\abs{
\brac{
\lambda^{-|\alpha|}
\vvd_\lambda^\alpha 
\int_{\R}  \bra{A_n[\mu(\z)\vp],e^{-it\mh}\mu(\e_j)\vp}dt
-
\Pa_\z^\alpha
\int_{\R}  \bra{A_n[\mu(\z)\vp],e^{-it\mh}\mu(\e_j)\vp}dt
}_{\z=\binom{\lambda}{\lambda}}
}
\\
&\lesssim_{|\alpha|}
\lambda.
\end{align*} 
Hence we see from Remark \ref{rem:prop:W} that 
\begin{align*}
&
\abs{
\brac{
i
\lambda^{-|\alpha|}
\vvd_\lambda^\alpha 
\bra{ 
\s(\mu(\z)\vp),
\mu(\e_j)\vp
}
}_{\z=\binom{\lambda}{\lambda}}
-
\int_{\R}  \bra{\wW_\alpha[\vp],e^{-it\mh}\mu(\e_j)\vp}dt
}
\lesssim_{|\alpha|}
\lambda,
\end{align*} 
which implies the desired estimate.
\end{proof}

\section{Solutions to the free Klein-Gordon equation}
In this section,
we study the massless limit of the solutions
$R_\vp,L_\vp$ and $M_\vp$ 
to the free Klein-Gordon equation
in order to determine the value of $\nn^{(\alpha)}(0)$.

For $n\in \Z$, 
let $J_n$ be the Bessel function of order $n$.
Recall that $J_0(0)=1$ and $\Px J_0=-J_1$.
Put $\vp\in \Sc$.
Fix $(t,x)\in \R \times \R$.
Then it is well known that 
\begin{align*}
M_\vp(t,x)
&=
\oi \sin(t\omega)\vp(x)
=
\frac{1}{2}
\int_{-t}^t J_0(\sqrt{t^2-y^2})\vp(x-y)dy,
\end{align*} 
which implies that 
\begin{align*}
\cos(t\omega)\vp(x)
&=\Pt \oi \sin(t\omega)\vp(x)
\\
&=\frac{1}{2}
\braa{
\vp(x-t)+\vp(x+t)
}
-
\frac{1}{2}
\int_{-t}^t J_1(\sqrt{t^2-y^2})
\frac{t}{\sqrt{t^2-y^2}}\vp(x-y)dy.
\end{align*}
Using integration by parts, 
we have
\begin{align*}
&
\oi \sin(t\omega)\Px \vp(x)
=
-\frac{1}{2}\int_{-t}^t J_0(\sqrt{t^2-y^2})
\frac{\Pa}{\Pa y}\vp(x-y)dy
\\
&=
\frac{1}{2}
\braa{
-\vp(x-t)+\vp(x+t)
}
+
\frac{1}{2}
\int_{-t}^t J_1(\sqrt{t^2-y^2})
\frac{y}{\sqrt{t^2-y^2}}\vp(x-y)dy.
\end{align*} 
Therefore,
we obtain
\begin{align*}
R_\vp(t,x)
=
\vp(x-t)
-
\frac{1}{2}
\int_{-t}^t J_1(\sqrt{t^2-y^2})
\frac{t+y}{\sqrt{t^2-y^2}}\vp(x-y)dy
\end{align*} 
and
\begin{align*}
L_\vp(t,x)
=
\vp(x+t)
-
\frac{1}{2}
\int_{-t}^t J_1(\sqrt{t^2-y^2})
\frac{t-y}{\sqrt{t^2-y^2}}\vp(x-y)dy.
\end{align*} 
We now put $m>0$.
Then we see that 
\begin{align}
R_{\vp,m}(t,x)
&=
R_{D_m\vp}(mt,mx)
\nonumber
\\
&=
\vp(x-t)
-
\frac{m}{2}
\int_{-t}^t J_1(m\sqrt{t^2-y^2})
\frac{t+y}{\sqrt{t^2-y^2}}\vp(x-y)dy.
\label{id:Rm}
\end{align} 
Similarly,
we have 
\begin{align*}
L_{\vp,m}(t,x)
=
\vp(x+t)
-
\frac{m}{2}
\int_{-t}^t J_1(m\sqrt{t^2-y^2})
\frac{t-y}{\sqrt{t^2-y^2}}\vp(x-y)dy
\end{align*} 
and
\begin{align}\label{id:Mm}
M_{\vp,m}(t,x)
=
\frac{m}{2}
\int_{-t}^t J_0(m\sqrt{t^2-y^2})\vp(x-y)dy.
\end{align}

We next give the decay estimates 
for $R_{\vp,m}$, 
$L_{\vp,m}$, 
and $M_{\vp,m}$. 
\begin{lem}\label{lem:RL:1}
Assume $\vp\in \Sc$ and $\supp \vp \subset [-1,1]$.
Then for any $m\in (0,1)$, 
the following inequalities hold true: 
\begin{align}
|R_{\vp,m}(t,x)|, 
|L_{\vp,m}(t,x)| \lesssim_\vp
1 \wedge (m^{-1}|t|^{-1/2}),
\quad &(t,x)\in \R \times \R,
\label{est:1:prop:RL:1}
\\
|M_{\vp,m}(t,x)| 
 \lesssim_\vp
m \wedge |t|^{-1/2},
\quad &(t,x)\in \R \times \R,
\label{est:2:prop:RL:1}
\\
|R_{\vp,m}(t,x)-\vp(x-t)|,
|L_{\vp,m}(t,x)-\vp(x+t)| \lesssim_\vp
m^2, 
\quad &|t| \le 2, \ x\in \R,
\label{est:3:prop:RL:1}
\\
|R_{\vp,m}(t,x) L_{\vp,m}(t,x)| \lesssim_\vp
m^{4/3}|t|^{-1/3}, 
\quad &|t|\ge 2, \ x\in \R.
\label{est:4:prop:RL:1}
\end{align} 
\end{lem}

\begin{proof}
Estimates (\ref{est:1:prop:RL:1}) and (\ref{est:2:prop:RL:1})
follow immediately from Proposition \ref{prop:decay}.

We repeatedly use the inequality 
$|J_1(r)|\le r(1+r^2)^{-3/4}$ ($r>0$).
If $|t|\le 2$, 
then we obtain 
\begin{align*}
|R_{\vp,m}(t,x)-\vp(x-t)|
&\lesssim
m
\int_{-|t|}^{|t|} 
\abs{  
\frac{J_1(m\sqrt{t^2-y^2})}{\sqrt{t^2-y^2}} 
}
\abs{ (t+y) \vp(x-y) } dy
\\
&\lesssim
m^2  
\int_{-|t|}^{|t|} 
\abs{ (t+y) \vp(x-y) } dy
 \lesssim_\vp
m^2. 
\end{align*} 
Hence (\ref{est:3:prop:RL:1}) holds true.

As for the proof of (\ref{est:4:prop:RL:1}), 
we have only to consider the case where 
$t \ge 2$ and $-t-1 \le x \le t+1$.
We fix $t\ge 2$.
If $-t-1\le x \le t-1$,
then we have $\vp(x-t)=0$ and 
\begin{align*}
|R_{\vp,m}(t,x)|
&\lesssim
m \int_{[-t,t]\cap [x-1,x+1]}
\frac{t+y}{\sqrt{t^2-y^2}}
\abs{ J_1(m\sqrt{t^2-y^2}) }
|\vp(x-y)|dy
\\
&\lesssim_\vp
m^2 
\int_{[-t,t]\cap [x-1,x+1]}
\frac{t+y}{
\braa{
1+m^2(t^2-y^2)
}^{3/4}
}
dy.
\end{align*}  
Since 
$(t+y)(1+m^2 (t^2-y^2))^{-3/4}$ 
is monotone increasing on $[-t,t]$,
we obtain 
\begin{align*}
|R_{\vp,m}(t,x)|
\lesssim_\vp
m^2 
\frac{t+x+1}{
\braa{
1+m^2(t^2-(x+1)^2)
}^{3/4}
}
\quad
(-t-1\le x \le t-1).
\end{align*}  
Similarly,
we see that
\begin{align*}
|L_{\vp,m}(t,x)|
\lesssim_\vp
m^2 
\frac{t-x+1}{
\braa{
1+m^2(t^2-(x-1)^2)
}^{3/4}
}
\quad
(-t+1\le x \le t+1).
\end{align*} 
It follows from (\ref{est:1:prop:RL:1}) that 
for any $x\in [0,t-2]$, 
\begin{align*}
&
|L_{\vp,m}(t,x) R_{\vp,m}(t,x)|
\\
&\lesssim_\vp
\braa{
\frac{
m^4 (t-x+1)(t+x+1) 
}{
\braa{
m^4 (t+x+1)(t+x-1)(t-x+1)(t-x-1)
}^{3/4}
}
}^{2/3}
\braa{
m^2 (t-x+1) \cdot 1
}^{1/3}
\\
&\lesssim_\vp
m^{4/3}
(t+x+1)^{2/3-1/2}
(t+x-1)^{-1/2}
(t-x+1)^{2/3-1/2+1/3}
(t-x-1)^{-1/2}
\\
&\lesssim_\vp
m^{4/3}
(t+x+1)^{1/6}
(t+x-1)^{-1/2}
\sqrt{\frac{t-x+1}{t-x-1}}
\\
&\lesssim_\vp
m^{4/3}
(2t-1)^{1/6}
(t-1)^{-1/2}
 \lesssim_\vp
m^{4/3} t^{-1/3}, 
\end{align*}
and that for any $x\in [t-2,t+1]$, 
\begin{align*}
|L_{\vp,m}(t,x) R_{\vp,m}(t,x)|
\lesssim_\vp
\braa{
m^2 \cdot m^{-1}|t|^{-1/2}
}^{2/3}
\braa{
m^2 \cdot 1
}^{1/3}
\lesssim_\vp
m^{4/3} t^{-1/3}. 
\end{align*} 
Hence (\ref{est:4:prop:RL:1}) holds true.
\end{proof}

Recall that we have defined functions
$\vp_+(t,x)=\vp(x-t)$ and $\vp_-(t,x)=\vp(x+t)$.
Formally, 
as $m\to +0$ we have 
$R_{\vp,m}\to \vp_+$, 
$L_{\vp,m}\to \vp_-$, 
and 
$M_{\vp,m}\to 0$.
The following two propositions mean 
that these limits actually hold true in some sense.  
\begin{prop}\label{prop:RL:2}
Let $n,\ell \in \Z$ satisft $n+\ell \ge 2$.
Assume $\vp\in \Sc$ and  
$\supp \vp \subset [-1,1]$.
Then it follows that for any $m\in (0,1)$, 
\begin{align*}
\abs{
\int_{\R \times \R} 
\braa{ R_{\vp,m} }^{   n+2}
\braa{ L_{\vp,m} }^{\ell+2}
d(t,x)
-
\int_{\R \times \R} 
\braa{\vp_+}^{   n+2}
\braa{\vp_-}^{\ell+2}
d(t,x)
}
\lesssim_\vp,
m^{1/2}.
\end{align*}
\end{prop}

\begin{proof}
It suffices to consider the case $\ell=0$.
Remark that 
the support of the function 
$(t,x)\mapsto \vp_+(t,x)^{   n+2}\vp_-(t,x)^{2}$
is included in 
$\{ (t,x)\in \R \times \R \,;\, |t|+|x| \le 1\}$.
By 
(\ref{est:1:prop:RL:1}),
(\ref{est:3:prop:RL:1}), 
(\ref{est:4:prop:RL:1}) 
and Proposition \ref{prop:decay}, 
we have 
\begin{align*}
&
\abs{
\int_{|t|\le 2}
\int_{\R}
R_{\vp,m}(t,x)^{   n+2}
L_{\vp,m}(t,x)^{     2}
dxdt
-
\int_{|t|\le 2}
\int_{\R}
\vp_+(t,x)^{   n+2}
\vp_-(t,x)^{     2}
dxdt
}
\\
&\lesssim
\abs{
\int_{-2}^2
\int_{-3}^3
\braa{
R_{\vp,m}(t,x)-\vp(x-t)
}
R_{\vp,m}(t,x)^{   n+1}
L_{\vp,m}(t,x)^{     2}
dxdt
}
+\cdots
\\
&\lesssim
\abs{
\int_{-2}^2
\int_{-3}^3
m^2 
1^{   n+1}
1^{     2}
dxdt
}
+\cdots
\lesssim
m^2 
\end{align*}
and 
\begin{align*}
&
\abs{
\int_{|t|\ge 2}
\int_{\R}
R_{\vp,m}(t,x)^{   n+2}
L_{\vp,m}(t,x)^{     2}
dxdt
}
\\
&\lesssim
\int_{|t|\ge 2}
\int_{\R}
\abs{
R_{\vp,m}(t,x)
L_{\vp,m}(t,x)
}^{27/20}
\braa{
\abs{
R_{\vp,m}(t,x)
}
+
\abs{
L_{\vp,m}(t,x)
}
}^{33/10+n-2}
dxdt
\\
&\lesssim_\vp
\int_{|t|\ge 2}
\braa{
m^{4/3}|t|^{-1/3}
}^{27/20}
\braa{
\nr{
R_{\vp,m}(t)
}_{33/10}
+
\nr{
L_{\vp,m}(t)
}_{33/10}
}^{33/10}
1^{n-2}
dt
\\
&\lesssim_\vp
\int_2^{\infty}
\braa{
m^{4/3}t^{-1/3}
}^{27/20}
\braa{
m^2 t
}^{-(1/2-10/33)(33/10)}
dt
=
m^{1/2}
\int_2^{\infty}
t^{-11/10}
dt
\lesssim
m^{1/2},
\end{align*} 
which completes the proof.
\end{proof}

\begin{prop}\label{prop:RL:3}
Put $n,\ell,k \in \Z$ with $n+\ell+k \ge 3$.
Assume $\vp\in \Sc$ and  
$\supp \vp \subset [-1,1]$.
Then it follows that for any $m\in (0,1)$, 
\begin{align*}
\abs{
\int_{\R \times \R} 
\braa{ R_{\vp,m} }^{   n+1}
\braa{ L_{\vp,m} }^{\ell+1}
\braa{ M_{\vp,m} }^{   k+1}
d(t,x)
}
&\lesssim_\vp
m^{1/2},
\\
\abs{
\int_{\R \times \R} 
\braa{ R_{\vp,m} }^{   n  }
\braa{ L_{\vp,m} }^{\ell  }
\braa{ M_{\vp,m} }^{   k+3}
d(t,x)
}
&\lesssim_\vp
m^{1/2}.
\end{align*}
\end{prop}

\begin{proof}
It suffices to show the case $\ell=k=0$.
By 
(\ref{est:1:prop:RL:1}),
(\ref{est:2:prop:RL:1}), 
(\ref{est:4:prop:RL:1}) 
and Proposition \ref{prop:decay}, 
we have 
\begin{align*}
&
\abs{
\int_{|t|\le 2}
\int_{\R}
R_{\vp,m}(t,x)^{n+1}
L_{\vp,m}(t,x)
M_{\vp,m}(t,x)
dxdt
}
\lesssim
\abs{
\int_{-2}^2
\int_{-3}^3
1^{n+1} \cdot 1 \cdot m
dxdt
}
\lesssim
m
\end{align*}
and 
\begin{align*}
&
\abs{
\int_{|t|\ge 2}
\int_{\R}
R_{\vp,m}(t,x)^{n+1}
L_{\vp,m}(t,x)
M_{\vp,m}(t,x)
dxdt
}
\\
&\lesssim
\int_{|t|\ge 2}
\int_{\R}
\abs{
M_{\vp,m}(t,x)
}
\abs{
R_{\vp,m}(t,x)
L_{\vp,m}(t,x)
}
\abs{
R_{\vp,m}(t,x)
}^{17/6+n-17/6}
dxdt
\\
&\lesssim_\vp
\int_{|t|\ge 2}
|t|^{-1/2}
m^{4/3}|t|^{-1/3}
\nr{
R_{\vp,m}(t)
}_{17/6}^{17/6}
1^{n-17/6}
dt
\\
&\lesssim_\vp
\int_2^{\infty}
m^{4/3}t^{-5/6}
\braa{
m^2 t
}^{-(1/2-6/17)(17/6)}
dt
=
m^{1/2}
\int_2^{\infty}
t^{-5/4}
dt
\lesssim
m^{1/2}.
\end{align*} 
Therefore, 
the first desired inequality holds true.
Similarly, 
we see that 
\begin{align*}
&
\abs{
\int_{|t|\le 2}
\int_{\R}
R_{\vp,m}(t,x)^n
M_{\vp,m}(t,x)^3
dxdt
}
\lesssim
\abs{
\int_{-2}^2
\int_{-3}^3
1^n \cdot m^3 
dxdt
}
\lesssim
m^3
\end{align*}
and 
\begin{align*}
&
\abs{
\int_{|t|\ge 2}
\int_{\R}
R_{\vp,m}(t,x)^n
M_{\vp,m}(t,x)^3
dxdt
}
\\
&\lesssim
\int_{|t|\ge 2}
\int_{\R}
\abs{
M_{\vp,m}(t,x)
}
\abs{
M_{\vp,m}(t,x)
}^2
\abs{
R_{\vp,m}(t,x)
}^{5/2}
\abs{
R_{\vp,m}(t,x)
}^{n-5/2}
dxdt
\\
&\lesssim_\vp
\int_{|t|\ge 2}
m
|t|^{-1}
\nr{
R_{\vp,m}(t)
}_{5/2}^{5/2}
1^{n-5/2}
dt
\\
&\lesssim_\vp
\int_2^{\infty}
m 
t^{-1}
\braa{
m^2 t
}^{-(1/2-2/5)(5/2)}
dt
=
m^{1/2}
\int_2^{\infty}
t^{-5/4}
dt
\lesssim
m^{1/2}.
\end{align*}
Hence we obtain the second desired inequality.
\end{proof}

We next consider some domains of $\R \times \R$ 
on which $R_{\vp,m}$, 
$L_{\vp,m}$, 
and $M_{\vp,m}$ do not decay sufficiently.
\begin{lem}\label{lem:RL:4}
Assume $\vp\in \Sc$, 
$\supp \vp \subset [-1,1]$ and $\vp>0$ on $(-1,1)$.
Then for any $m\in (0,1/2)$,
the following estimates hold true: 
\begin{enumerate}[(1)]
  \item 
If $0\le t \lesssim_\vp m^{-2}$ 
and $t-1/2 \le x \le t+1/2$, 
then 
$R_{\vp,m}(t,x) \gtrsim_\vp 1$.
  \item 
If $0\le t \lesssim_\vp m^{-2}$ 
and $-t-1/2 \le x \le -t+1/2$, 
then 
$L_{\vp,m}(t,x) \gtrsim_\vp 1$.
  \item 
If $2\le t \le 1/(2m^2)$ 
and $t-1/2 \le x \le t+1/2$, 
then 
$M_{\vp,m}(t,x) \gtrsim_\vp m$.
\end{enumerate}
\end{lem}

\begin{proof}
Since $|J_1(r)/r|\le 1/2$ ($r> 0$),
we see from (\ref{id:Rm}) that 
if $t\ge 0$ and $t-1/2 \le x \le t+1/2$, 
then 
\begin{align*}
R_{\vp,m}(t,x)
&\ge 
\vp(x-t)
-
\frac{m^2}{2}
\int_{-t}^t 
\abs{J_1(m\sqrt{t^2-y^2})
\frac{t+y}{m\sqrt{t^2-y^2}}}\vp(x-y)dy
\\
&\ge 
\vp(x-t)-\frac{m^2}{4}
\int_{[-t,t]\cap [x-1,x+1]} |t+y|\vp(x-y)dy
\\
&\ge
\vp(x-t)-\frac{m^2 t}{2}
\int_{-1}^1 \vp(y)dy
\ge 
\inf_{|x|\le 1/2}\vp(x)
-\frac{m^2 t}{2}\nr{\vp}_1.
\end{align*} 
It follows that 
\begin{align*}
R_{\vp,m}(t,x) \ge \frac{1}{2}\inf_{|x|\le 1/2}\vp(x)
\end{align*} 
provided that 
$0\le t \le \inf_{|x|\le 1/2}|\vp(x)|/(m^2 \nr{\vp}_1)$ 
and 
$t-1/2 \le x \le t+1/2$.
Hence we obtain (1).
Similarly, 
we have (2).

Henceforth,
we set $2\le t \le 1/(2m^2)$ 
and $t-1/2 \le x \le t+1/2$.
Then we obtain as follows:
\begin{itemize}
  \item 
$[-t,t]\cap[x-1,x+1]=[t-2,t]\cap[x-1,x+1]$.
  \item 
If $y\in [t-2,t]$,
then $0\le m\sqrt{t^2-y^2} \le \sqrt{2}$.
  \item 
$[t-2,t]\cap[x-1,x+1]\supset [x-1,x-1/2]$.
\end{itemize}
Since $J_0(\sqrt{2})\ge 1/2$,
by (\ref{id:Mm}) we obtain 
\begin{align*}
M_{\vp,m}(t,x)
&=
\frac{m}{2}
\int_{[-t,t]\cap [x-1,x+1]} J_0(m\sqrt{t^2-y^2})\vp(x-y)dy
\\
&=
\frac{m}{2}
\int_{[t-2,t]\cap[x-1,x+1]} J_0(m\sqrt{t^2-y^2})\vp(x-y)dy
\\
&\ge
\frac{m}{2}
\int_{[t-2,t]\cap[x-1,x+1]} \frac{1}{2}\vp(x-y)dy
\\
&\ge
\frac{m}{4}\int^{x-1/2}_{x-1}\vp(x-y)dy
=
\frac{m}{4}\int_{1/2}^{1}\vp(y)dy,
\end{align*} 
which completes the proof of (3).
\end{proof}

We finally give some properties 
for integrations 
which do not appear in Propositions 
\ref{prop:RL:2} and \ref{prop:RL:3}.

\begin{prop}\label{prop:RL:5}
Set $\vp\in \Sc$, 
$m>0$ and $N\in \N$ with $N\ge 5$.
\begin{enumerate}[(1)]
  \item 
We have 
\begin{align*}
  \int_{\R \times \R} \braa{ R_{\vp,m} }^N  L_{\vp,m}    d(t,x)
=
N \int_{\R \times \R} \braa{ R_{\vp,m} }^{N-1} \braa{ M_{\vp,m} }^2 d(t,x) 
\end{align*} 
and 
\begin{align*}
  \int_{\R \times \R} \braa{ L_{\vp,m} }^N     R_{\vp,m}   d(t,x)
=
N \int_{\R \times \R} \braa{ L_{\vp,m} }^{N-1} \braa{ M_{\vp,m} }^2 d(t,x). 
\end{align*} 
  \item 
If $\vp$ is even,
then 
\begin{align*}
\int_{\R \times \R} \braa{ R_{\vp,m} }^N M_{\vp,m} d(t,x)
=
\int_{\R \times \R} \braa{ L_{\vp,m} }^N M_{\vp,m} d(t,x)
=
0. 
\end{align*} 
  \item 
If $m<1$, 
then 
\begin{align*}
\abs{
\int_{\R \times \R} \braa{ R_{\vp,m} }^{N-1} \braa{ M_{\vp,m} }^2 d(t,x) 
},\ 
\abs{
\int_{\R \times \R} \braa{ L_{\vp,m} }^{N-1} \braa{ M_{\vp,m} }^2 d(t,x) 
}
\lesssim_{N,\vp} 1.
\end{align*} 
  \item 
Assume that $N$ is odd.
We also assume $m<1/2$, 
$\supp \vp \subset [-1,1]$ and $\vp>0$ on $(-1,1)$.
It follows that
\begin{align*}
\int_{\R \times \R} \braa{ R_{\vp,m} }^{N+1} d(t,x),
\ 
\int_{\R \times \R} \braa{ L_{\vp,m} }^{N+1} d(t,x) 
\gtrsim_{\vp} m^{-2}
\end{align*} 
and  
\begin{align*}
\int_{\R \times \R} \braa{ R_{\vp,m} }^{N-1} \braa{ M_{\vp,m} }^2 d(t,x),
\ 
\int_{\R \times \R} \braa{ L_{\vp,m} }^{N-1} \braa{ M_{\vp,m} }^2 d(t,x) 
\gtrsim_{\vp} 1.
\end{align*} 
\end{enumerate}
\end{prop}

\begin{proof}
\begin{enumerate}[(1)]
  \item 
Since $L_{\vp,m}=m^{-1}(\Pt+\Px)M_{\vp,m}$,
we see from integration by parts that 
\begin{align*}
&
  \int_{\R \times \R} R_{\vp,m}(t,x)^N     L_{\vp,m}(t,x)   d(t,x)
=-m^{-1}
\int_{\R \times \R} M_{\vp,m}(t,x) (\Pt+\Px)R_{\vp,m}(t,x)^N  d(t,x).
\end{align*} 
Using the identity 
\begin{align*}
(\Pt+\Px)R_{\vp,m}
&=
m^{-1}(\Pt+\Px)(\Pt-\Px)M_{\vp,m}
=
m^{-1}(\Pt^2-\Px^2)M_{\vp,m}
\\
&=
-m^{-1}(\omega_m^2+\Px^2)M_{\vp,m}
=
-m M_{\vp,m},
\end{align*}
we have 
\begin{align*}
&
  \int_{\R \times \R} R_{\vp,m}(t,x)^N     L_{\vp,m}(t,x)   d(t,x)
\\
&=
-m^{-1} N
\int_{\R \times \R} M_{\vp,m}(t,x) R_{\vp,m}(t,x)^{N-1} (\Pt+\Px)R_{\vp,m}(t,x) d(t,x)
\\
&=
N
\int_{\R \times \R} R_{\vp,m}(t,x)^{N-1} M_{\vp,m}(t,x)^2 d(t,x),
\end{align*}  
which completes the proof of the first desired identity.
The second one is shown similarly. 
  \item 
If $\vp$ is even, 
then the desired identity
follows from the fact that 
for any $(t,x)\in \R \times \R$,
$R_{\vp,m}(-t,-x)=R_{\vp,m}(t,x)$,
$L_{\vp,m}(-t,-x)=L_{\vp,m}(t,x)$
and 
$M_{\vp,m}(-t,-x)=-M_{\vp,m}(t,x)$.
  \item 
By Lemma \ref{lem:RL:1}(1), 
the H\"older inequality,
and Proposition \ref{prop:St}(2), 
we have  
\begin{align*}
&
\abs{
\int_{\R \times \R} R_{\vp,m}(t,x)^{N-1} M_{\vp,m}(t,x)^2 d(t,x) 
}
\\
&\lesssim
\int_{\R \times \R} 
\abs{R_{\vp,m}(t,x)}^4 \abs{M_{\vp,m}(t,x)}^2 \abs{R_{\vp,m}(t,x)}^{N-5}
d(t,x)
\\
&\lesssim
\nr{ R_{\vp,m} }_{L^{20/3}L^5}^4  
\nr{ M_{\vp,m} }_{L^5L^{10}}^2
1^{N-5}
\lesssim_\vp 
\braa{ m^{-3/10} }^4  \braa{ m^{3/5} }^2
\lesssim 1.  
\end{align*} 
Similarly, 
we obtain another desired estimate.
  \item 
Using Lemma \ref{lem:RL:4},
we immediately obtain the required property.
\end{enumerate}
\end{proof}

\section{Proof of Theorem \ref{thm:main:1}}
In this section,
we establish the first main theorem. 
Let $N\in \N$ satisfy $N\ge 5$.
Suppose that Hypotheses (H1) and (H3) hold, 
and that $\nn^{(\beta)}(0)$ with $|\beta|\le N-4$ is known.
Choose an even function 
$\vp \in \Sc \cap B_{\eta_0}H^1$ 
such that $\supp\vp \subset [-1,1]$, 
and $\vp>0$ on $(-1,1)$.
Set $j=1,2$.
Let $\alpha=(\alpha_1,\alpha_2,\alpha_3,\alpha_4)\in \Z^4$ 
satisfy $|\alpha|=N$  
and $\Ajk{j}{k}$ for some $k=1,\cdots,6$.
Put $m\in (0,1/2)$
and 
$0<\lambda <1/(N+1)$.
Note that 
$I_{\alpha,j}[m^{1/2}D_m \vp](1,\lambda)$ 
is known, 
and that 
$\nr{m^{1/2}D_m \vp}\le \nr{\vp}$.
It follows from Lemma \ref{lem:V} that 
\begin{align}\label{est:V-m}
\abs{
I_{\alpha,j}[m^{1/2}D_m \vp](1,\lambda)
-
\int_{\R}  \bra{V_\alpha[m^{1/2}D_m \vp],e^{-it\mh}\mu(\e_j)m^{1/2} D_m\vp}dt
}
\lesssim_N
\lambda.
\end{align}
Since 
$R_{D_m \vp}(t,x)=R_{\vp,m}(t/m,x/m)$,
$L_{D_m \vp}(t,x)=L_{\vp,m}(t/m,x/m)$,
and 
$M_{D_m \vp}(t,x)=M_{\vp,m}(t/m,x/m)$
($(t,x)\in \R \times \R$),
by 
the definitions of $I_{\alpha,j}$ and $V_\alpha$, 
and (\ref{id:U(t):2}), 
we have 
\begin{align*}
I_{\alpha,j}[m^{1/2}D_m \vp](1,\lambda)
=
m^{2+(N+1)/2} I_{\alpha,j}[\vp](m,\lambda)
\end{align*} 
and 
\begin{align*}
\int_{\R}  \bra{V_\alpha[m^{1/2}D_m \vp],e^{-it\mh}m^{1/2}\mu(\e_j)D_m\vp}dt
=
m^2 J_{\alpha,j}[m^{1/2}\vp](m)
=
m^{2+(N+1)/2} J_{\alpha,j}[\vp](m).
\end{align*} 
Here, 
we have defined 
\begin{align}\label{def:J1}
&
J_{\alpha,1}[\vp](m)
\nonumber
\\
&=
\sum_{\gamma\le \alpha}
\frac{\alpha!}{\gamma!(\alpha-\gamma)!}
\left\{
\int_{\R\times \R}  
\brad{ R_{\vp,m}, L_{\vp,m} }^{\alpha-\gamma}
R_{\vp,m}
\brad{iM_{\vp,m},iM_{\vp,m} }^\gamma
d(t,x)\, 
\n_1^{(\alpha-\gamma+\widetilde{\gamma})}(0)
\right.
\nonumber
\\
&\quad -i
\left.
\int_{\R\times \R}  
\brad{ R_{\vp,m}, L_{\vp,m} }^{\alpha-\gamma}
\brad{iM_{\vp,m},iM_{\vp,m} }^\gamma
M_{\vp,m}
d(t,x)\, 
\n_2^{(\alpha-\gamma+\widetilde{\gamma})}(0)
\right\}
\end{align} 
and 
\begin{align*}
&
J_{\alpha,2}[\vp](m)
\nonumber
\\
&=
\sum_{\gamma\le \alpha}
\frac{\alpha!}{\gamma!(\alpha-\gamma)!}
\left\{
\int_{\R\times \R}  
\brad{ R_{\vp,m}, L_{\vp,m} }^{\alpha-\gamma}
L_{\vp,m}
\brad{iM_{\vp,m},iM_{\vp,m} }^\gamma
d(t,x)\, 
\n_2^{(\alpha-\gamma+\widetilde{\gamma})}(0)
\right.
\nonumber
\\
&\quad -i
\left.
\int_{\R\times \R}  
\brad{ R_{\vp,m}, L_{\vp,m} }^{\alpha-\gamma}
\brad{iM_{\vp,m},iM_{\vp,m} }^\gamma
M_{\vp,m}
d(t,x)\, 
\n_1^{(\alpha-\gamma+\widetilde{\gamma})}(0)
\right\}.
\end{align*} 
By (\ref{est:V-m}), 
we obtain 
\begin{align}\label{est:V-m:2}
\abs{
I_{\alpha,j}[\vp](m,\lambda)
-
J_{\alpha,j}[\vp](m)
}
\lesssim_N
m^{-2-(N+1)/2}
\lambda.
\end{align}
Recall the definitions of 
$d_{\vp,\alpha,+}$,
$d_{\vp,\alpha,-}$,
$d_{\vp,\alpha,m,R}$, 
and  
$d_{\vp,\alpha,m,M}$ given in Section 1.
As we mention in Remark \ref{rem:isp-notation}(7),
under the assumption that $\vp$ is even, 
we have (\ref{id:R-L}). 
We now divide the proof of (\ref{est:thm:main:1}) with $j=1$ 
into five cases. 
Note that the case $j=2$ can be shown similarly.

\subsection*{Case 1.}
Suppose that $\alpha \in \Ajk{1}{1}$.
It follows from Proposition \ref{prop:RL:2} that 
\begin{align*}
&
\left|
\int_{\R\times \R}  
\brad{ R_{\vp,m}, L_{\vp,m} }^{\alpha}
R_{\vp,m}
d(t,x)\, 
\n_1^{(\alpha)}(0)
\right.
\\
&
\left. 
-
\int_{\R \times \R}
\brad{\vp_+, \vp_-}^{\alpha+(0,1,0,0)}
d(t,x)\,
\n_1^{(\alpha)}(0)
\right|
\lesssim_{N,\vp} m^{1/2}.
\end{align*} 
We also see from (\ref{def:J1}) 
and Proposition \ref{prop:RL:3} that 
\begin{align*}
&
\abs{
J_{\alpha,1}[\vp](m)
-
\int_{\R\times \R}  
\brad{ R_{\vp,m}, L_{\vp,m} }^{\alpha}
R_{\vp,m}
d(t,x)\, 
\n_1^{(\alpha)}(0)
}
\\
&
\lesssim_N
\sum_{
\substack{
\gamma\le \alpha
\\
\gamma\neq \mathbf{0}
}
}
\abs{
\int_{\R\times \R}  
\brad{ R_{\vp,m}, L_{\vp,m} }^{\alpha-\gamma}
R_{\vp,m}
\brad{iM_{\vp,m},iM_{\vp,m} }^\gamma
d(t,x)\, 
\n_1^{(\alpha-\gamma+\widetilde{\gamma})}(0)
}
\\
& \quad+
\sum_{\gamma\le \alpha}
\abs{
\int_{\R\times \R}  
\brad{ R_{\vp,m}, L_{\vp,m} }^{\alpha-\gamma}
\brad{iM_{\vp,m},iM_{\vp,m} }^\gamma
M_{\vp,m}
d(t,x)\, 
\n_2^{(\alpha-\gamma+\widetilde{\gamma})}(0)
}
\lesssim_{N,\vp} m^{1/2}.
\end{align*} 
By (\ref{est:V-m:2}), 
we obtain 
\begin{align*}
\abs{
K_{\alpha,1}[\vp](m,\lambda)-\n_1^{(\alpha)}(0)
}
&=
\abs{
\frac{
I_{\alpha,1}[\vp](m,\lambda)
}{
d_{\vp,\alpha,+}
}
-
\n_1^{(\alpha)}(0)
}
\\
&\le 
\abs{
\frac{
I_{\alpha,1}[\vp](m,\lambda)
}{
d_{\vp,\alpha,+}
}
-
\frac{
J_{\alpha,1}[\vp](m)
}{
d_{\vp,\alpha,+}
}
}
+
\abs{
\frac{
J_{\alpha,1}[\vp](m)
}{
d_{\vp,\alpha,+}
}
-
\n_1^{(\alpha)}(0)
}
\\
&\lesssim_{N,\vp} m^{-2-(N+1)/2}\lambda + m^{1/2}.
\end{align*}

\subsection*{Case 2.}
Suppose that $\alpha \in \Ajk{1}{2}$.
It follows from (\ref{def:J1}) and Propositions 
\ref{prop:RL:3} and \ref{prop:RL:5} that 
\begin{align*}
&
\abs{
J_{\alpha,1}[\vp](m)
-
\int_{\R\times \R}  
\braa{R_{\vp,m}}^{|\alpha|+1}
d(t,x)\, 
\n_1^{(\alpha)}(0)
}
\\
&\lesssim_N
\sum_{
\substack{
\gamma\le \alpha
\\
\gamma\neq  \mathbf{0}
}
}
\abs{
\int_{\R\times \R}  
\braa{R_{\vp,m}}^{|\alpha|+1-|\gamma|}
\braa{M_{\vp,m}}^{|\gamma|}
d(t,x)\, 
\n_1^{(\alpha-\gamma+\widetilde{\gamma})}(0)
}
\\
& \quad +
\sum_{\gamma\le \alpha}
\abs{
\int_{\R\times \R}  
\braa{R_{\vp,m}}^{|\alpha|-|\gamma|}
\braa{M_{\vp,m}}^{|\gamma|+1}
d(t,x)\, 
\n_2^{(\alpha-\gamma+\widetilde{\gamma})}(0)
}
\\
&\lesssim_{N,\vp} 1.
\end{align*} 
By (\ref{est:V-m:2}) and Proposition \ref{prop:RL:5}(4),
we obtain 
\begin{align*}
\abs{
K_{\alpha,1}[\vp](m,\lambda)-\n_1^{(\alpha)}(0)
}
=
\abs{
\frac{
I_{\alpha,1}[\vp](m,\lambda)
}{
d_{\vp,\alpha,m,R}
}
-
\n_1^{(\alpha)}(0)
}
\lesssim_{N,\vp} m^{-(N+1)/2}\lambda + m^2.
\end{align*}

\subsection*{Case 3.}
Suppose that $\alpha \in \Ajk{1}{3}$.
It follows from Proposition \ref{prop:RL:2} that 
\begin{align*}
&
\left|
\int_{\R\times \R}  
\brad{ R_{\vp,m}, L_{\vp,m} }^{\alpha}
R_{\vp,m}
d(t,x)\, 
\n_1^{(\alpha)}(0)
\right.
\\
&
\left. 
-
\int_{\R \times \R}
\brad{\vp_+, \vp_-}^{\alpha+(0,1,0,0)}
d(t,x)\,
\n_1^{(\alpha)}(0)
\right|
\lesssim_{N,\vp} m^{1/2}.
\end{align*} 
When $\gamma=(0,0,\alpha_3,\alpha_4)$,
we have 
\begin{align*}
&
\frac{\alpha!}{\gamma!(\alpha-\gamma)!}
\int_{\R\times \R}  
\brad{ R_{\vp,m}, L_{\vp,m} }^{\alpha-\gamma}
R_{\vp,m}
\brad{iM_{\vp,m},iM_{\vp,m} }^\gamma
d(t,x)\, 
\n_1^{(\alpha-\gamma+\widetilde{\gamma})}(0)
\\
&=
i^{\alpha_3-\alpha_4}
\int_{\R\times \R}  
\braa{ R_{\vp,m} }^{N-1}
\braa{ M_{\vp,m} }^2
d(t,x)\, 
\n_1^{(\alpha_1+\alpha_3,\alpha_2+\alpha_4,0,0)}(0).
\end{align*} 
We also obtain by Proposition \ref{prop:RL:3}, 
\begin{align*}
&
\sum_{
\substack{
\gamma\le \alpha
\\
\gamma\neq \mathbf{0}, (0,0,\alpha_3,\alpha_4)
}
}
\abs{
\int_{\R\times \R}  
\brad{ R_{\vp,m}, L_{\vp,m} }^{\alpha-\gamma}
R_{\vp,m}
\brad{iM_{\vp,m},iM_{\vp,m} }^\gamma
d(t,x)\, 
\n_1^{(\alpha-\gamma+\widetilde{\gamma})}(0)
}
\\
& \quad+
\sum_{\gamma\le \alpha}
\abs{
\int_{\R\times \R}  
\brad{ R_{\vp,m}, L_{\vp,m} }^{\alpha-\gamma}
\brad{iM_{\vp,m},iM_{\vp,m} }^\gamma
M_{\vp,m}
d(t,x)\, 
\n_2^{(\alpha-\gamma+\widetilde{\gamma})}(0)
}
\\
&\lesssim_{N,\vp} m^{1/2}.
\end{align*} 
Therefore, 
we see from (\ref{def:J1}),
(\ref{est:V-m:2}),  
Proposition \ref{prop:RL:5} and Case 2 that 
\begin{align*}
&
\left|
I_{\alpha,1}[\vp](m,\lambda)
-
\int_{\R \times \R}
\brad{\vp_+, \vp_- }^{\alpha+(0,1,0,0)}
d(t,x)\, \n_1^{(\alpha)}(0)
\right.
\\
& 
\left.
\quad -
(-1)^{(\alpha_3-\alpha_4)/2}
\int_{\R\times \R}  
\braa{ R_{\vp,m} }^{N-1}
\braa{ M_{\vp,m} }^2
d(t,x)
\frac{
I_{(\alpha_1+\alpha_3,\alpha_2+\alpha_4,0,0),1}[\vp](m,\lambda)
}{
\int_{\R\times \R}  
\braa{ R_{\vp,m} }^{N+1}
d(t,x)
}
\right|
\\
&\lesssim_{N,\vp}
m^{-2-(N+1)/2}\lambda + m^{1/2}
+\abs{
\frac{
I_{(\alpha_1+\alpha_3,\alpha_2+\alpha_4,0,0),1}[\vp](m,\lambda)
}{
\int_{\R\times \R}  
\braa{ R_{\vp,m} }^{N+1}
d(t,x)
}
-
\n_1^{(\alpha_1+\alpha_3,\alpha_2+\alpha_4,0,0)}(0)
}
\\
&\lesssim_{N,\vp}
m^{-2-(N+1)/2}\lambda + m^{1/2}+m^{-(N+1)/2}\lambda + m^2
\lesssim 
m^{-2-(N+1)/2}\lambda + m^{1/2},
\end{align*} 
which implies that 
\begin{align*}
&
\abs{
K_{\alpha,1}[\vp](m,\lambda)-\n_1^{(\alpha)}(0)
}
\\
&=
\left|
\frac{
I_{\alpha,1}[\vp](m,\lambda)
}{
d_{\vp,\alpha,+}
}
-
(-1)^{(\alpha_3-\alpha_4)/2}
\frac{
d_{\vp,\alpha,m,M}
\, 
I_{(\alpha_1+\alpha_3,\alpha_2+\alpha_4,0,0),1}[\vp](m,\lambda)
}{
d_{\vp,\alpha,+} 
d_{\vp,\alpha,m,R}
}
-
\n_1^{(\alpha)}(0)
\right|
\\
&\lesssim_{N,\vp}
m^{-2-(N+1)/2}\lambda + m^{1/2}.
\end{align*}

\subsection*{Case 4.}
Suppose that $\alpha \in \Ajk{1}{4}$.
It follows from Proposition \ref{prop:RL:2} that 
\begin{align*}
&
\left|
\int_{\R\times \R}  
\brad{ R_{\vp,m}, L_{\vp,m} }^{\alpha}
R_{\vp,m}
d(t,x)\, 
\n_1^{(\alpha)}(0)
\right.
\\
&
\left. 
-
\int_{\R \times \R}
\brad{\vp_+, \vp_-}^{\alpha+(0,1,0,0)}
d(t,x)\,
\n_1^{(\alpha)}(0)
\right|
\lesssim_{N,\vp} m^{1/2}.
\end{align*}
When $\gamma=(\alpha_1,\alpha_2,0,0)$,
we have 
\begin{align*}
&
-i
\frac{\alpha!}{\gamma!(\alpha-\gamma)!}
\int_{\R\times \R}  
\brad{ R_{\vp,m}, L_{\vp,m} }^{\alpha-\gamma}
\brad{iM_{\vp,m},iM_{\vp,m} }^\gamma
M_{\vp,m}
d(t,x)\, 
\n_2^{(\alpha-\gamma+\widetilde{\gamma})}(0)
\\
&=
-i^{1+\alpha_1-\alpha_2}
\int_{\R\times \R}  
\braa{ L_{\vp,m} }^{N-1}
\braa{ M_{\vp,m} }^2
d(t,x)\, 
\n_2^{(0,0,\alpha_1+\alpha_3,\alpha_2+\alpha_4)}(0).
\end{align*} 
We also obtain by Proposition \ref{prop:RL:3}, 
\begin{align*}
&
\sum_{
\substack{
\gamma\le \alpha
\\
\gamma\neq \mathbf{0}
}
}
\abs{
\int_{\R\times \R}  
\brad{ R_{\vp,m}, L_{\vp,m} }^{\alpha-\gamma}
R_{\vp,m}
\brad{iM_{\vp,m},iM_{\vp,m} }^\gamma
d(t,x)\, 
\n_1^{(\alpha-\gamma+\widetilde{\gamma})}(0)
}
\\
& \quad+
\sum_{
\substack{
\gamma\le \alpha
\\
\gamma\neq (\alpha_1,\alpha_2,0,0)
}
}
\abs{
\int_{\R\times \R}  
\brad{ R_{\vp,m}, L_{\vp,m} }^{\alpha-\gamma}
\brad{iM_{\vp,m},iM_{\vp,m} }^\gamma
M_{\vp,m}
d(t,x)\, 
\n_2^{(\alpha-\gamma+\widetilde{\gamma})}(0)
}
\\
&\lesssim_{N,\vp} m^{1/2}.
\end{align*} 
Therefore, 
we see from (\ref{def:J1}), 
(\ref{est:V-m:2}),  
Proposition \ref{prop:RL:5} and Case 2 that 
\begin{align*}
&
\left|
I_{\alpha,1}[\vp](m,\lambda)
-
\int_{\R \times \R}
\brad{\vp_+, \vp_- }^{\alpha+(0,1,0,0)}
d(t,x)\, \n_1^{(\alpha)}(0)
\right.
\\
& 
\left.
\quad -
(-1)^{(1+\alpha_1-\alpha_2)/2}
\int_{\R\times \R}  
\braa{ L_{\vp,m} }^{N-1}
\braa{ M_{\vp,m} }^2
d(t,x)
\frac{
I_{(0,0,\alpha_1+\alpha_3,\alpha_2+\alpha_4),2}[\vp](m,\lambda)
}{
\int_{\R\times \R}  
\braa{ L_{\vp,m} }^{N+1}
d(t,x)
}
\right|
\\
&\lesssim_{N,\vp}
m^{-2-(N+1)/2}\lambda + m^{1/2}
+
\abs{
\frac{
I_{(0,0,\alpha_1+\alpha_3,\alpha_2+\alpha_4),2}[\vp](m,\lambda)
}{
\int_{\R\times \R}  
\braa{ L_{\vp,m} }^{N+1}
d(t,x)
}
-
\n_2^{(0,0,\alpha_1+\alpha_3,\alpha_2+\alpha_4)}(0)
}
\\
&\lesssim_{N,\vp}
m^{-2-(N+1)/2}\lambda + m^{1/2}+m^{-(N+1)/2}\lambda + m^2
\lesssim 
m^{-2-(N+1)/2}\lambda + m^{1/2},
\end{align*} 
which implies that 
\begin{align*}
&
\abs{
K_{\alpha,1}[\vp](m,\lambda)-\n_1^{(\alpha)}(0)
}
\\
&=
\left|
\frac{
I_{\alpha,1}[\vp](m,\lambda)
}{
d_{\vp,\alpha,+}
}
-
(-1)^{(1+\alpha_1-\alpha_2)/2}
\frac{ 
d_{\vp,\alpha,m,M}\, 
I_{(0,0,\alpha_1+\alpha_3,\alpha_2+\alpha_4),2}[\vp](m,\lambda)
}{
d_{\vp,\alpha,+}
d_{\vp,\alpha,m,R}
}
-
\n_1^{(\alpha)}(0)
\right|
\\
&\lesssim_{N,\vp}
m^{-2-(N+1)/2}\lambda + m^{1/2}.
\end{align*}

\subsection*{Case 5.}
Suppose that $N$ is odd, 
and that $\alpha_1=(N+1)/2$ and $\alpha_2=(N-1)/2$.
We set 
$\ap=(\alpha_1-1,\alpha_2,1,0)$,
$\app=(\alpha_1,\alpha_2-1,0,1)$,
\begin{align*}
b_{\ap }&:=J_{\ap,1}[\vp](m),
\\
b_{\app}&:=J_{\app,1}[\vp](m),
\end{align*} 
and $d=d_{\vp,\alpha,m,M}$.
In order to calculate $b_{\ap }$, 
we define 
\begin{align*}
r_{\ap }
&=
\sum_{
\substack{
\gamma\le \ap
\\
\gamma\neq \mathbf{0},(1,0,1,0),(0,1,1,0)
}
}
\bigg\{
\frac{\ap!}{\gamma!(\ap-\gamma)!}
\\
&\qquad \times \int_{\R\times \R}  
\brad{ R_{\vp,m}, L_{\vp,m} }^{\ap-\gamma}
R_{\vp,m}
\brad{iM_{\vp,m},iM_{\vp,m} }^\gamma
d(t,x)\, 
\n_1^{(\ap-\gamma+\widetilde{\gamma})}(0)
\bigg\}
\\
&\quad -i
\sum_{
\substack{
\gamma\le \ap
\\
\gamma\neq (0,0,1,0)
}
}
\bigg\{
\frac{\ap!}{\gamma!(\ap-\gamma)!}
\\
&\qquad \times \int_{\R\times \R}  
\brad{ R_{\vp,m}, L_{\vp,m} }^{\ap-\gamma}
\brad{iM_{\vp,m},iM_{\vp,m} }^\gamma
M_{\vp,m}
d(t,x)\, 
\n_2^{(\ap-\gamma+\widetilde{\gamma})}(0)
\bigg\}.
\end{align*}
It follows from Proposition \ref{prop:RL:2} that 
$|r_{\ap }| \lesssim_{N,\vp} m^{1/2}$.
By (\ref{def:J1}), 
and Proposition \ref{prop:RL:5}(1) and (4),
we obtain 
\begin{align*}
b_{\ap }-r_{\ap }
&=
\int_{\R\times \R}  
\braa{ R_{\vp,m} }^N
L_{\vp,m}
d(t,x)\, 
\n_1^{(\ap)}(0)
\\
&\quad +
\int_{\R\times \R}  
\braa{ R_{\vp,m} }^{N-1}
\braa{ M_{\vp,m} }^2
d(t,x)\, 
\n_2^{(\alpha)}(0)
\\
&\quad -
(\alpha_1 -1)
\int_{\R\times \R}  
\braa{ R_{\vp,m} }^{N-1}
\braa{ M_{\vp,m} }^2
d(t,x)\, 
\n_1^{(\ap)}(0)
\\
&\quad +
\alpha_2 
\int_{\R\times \R}  
\braa{ R_{\vp,m} }^{N-1}
\braa{ M_{\vp,m} }^2
d(t,x)\, 
\n_1^{(\app)}(0)
\\
&=
d
\braa{
         \n_2^{(\alpha)}(0)
+
(\alpha_2+1) \n_1^{(\ap)}(0)
+
\alpha_2 \n_1^{(\app)}(0)
}
\end{align*}
and 
\begin{align*}
\abs{
\frac{b_{\ap }}{d}
-\braa{
         \n_2^{(\alpha)}(0)
+
(\alpha_2+1) \n_1^{(\ap)}(0)
+
 \alpha_2 \n_1^{(\app)}(0)
}
}
 \lesssim_{N,\vp} m^{1/2}.
\end{align*} 
Similarly, 
we have 
\begin{align*}
\abs{
\frac{b_{\app}}{d}
-\braa{
-        \n_2^{(\alpha)}(0)
+
 \alpha_1 \n_1^{(\ap)}(0)
+
 (\alpha_1+1) \n_1^{(\app)}(0)
}
}
 \lesssim_{N,\vp} m^{1/2},
\end{align*} 
which implies that 
\begin{align*}
\abs{
\binom{b_{\ap }/d}{b_{\app}/d}
-
\Miijjj{1}{\alpha_2+1}{\alpha_2}{-1}{\alpha_1}{\alpha_1+1}
\Miii{\n_2^{(\alpha)}(0)}{\n_1^{(\ap)}(0)}{\n_1^{(\app)}(0)}
}
 \lesssim_{N,\vp} m^{1/2}.
\end{align*} 
We now use Hypothesis (H3).
It follows that $\n_1^{(\ap)}(0)=\n_1^{(\app)}(0)$ and 
\begin{align*}
\abs{
\binom{b_{\ap }/d}{b_{\app}/d}
-
\Miijj{1}{N}{-1}{N+2}
\binom{\n_2^{(\alpha)}(0)}{\n_1^{(\ap)}(0)}
}
 \lesssim_{N,\vp} m^{1/2},
\end{align*} 
and hence that 
\begin{align*}
\abs{
\frac{1}{2N+2}
\Miijj{N+2}{-N}{1}{1}
\binom{b_{\ap }/d}{b_{\app}/d}
-
\binom{\n_2^{(\alpha)}(0)}{\n_1^{(\ap)}(0)}
}
 \lesssim_{N,\vp} m^{1/2}.
\end{align*} 
We see from (\ref{est:V-m:2})
and Proposition \ref{prop:RL:5}(4) that 
\begin{align*} 
\abs{
\frac{1}{2N+2}
\Miijj{N+2}{-N}{1}{1}
\binom{
I_{\ap,1}[\vp](m,\lambda)
/(d)}{
I_{\app,1}[\vp](m,\lambda)
/(d)}
-
\binom{\n_2^{(\alpha)}(0)}{\n_1^{(\ap)}(0)}
}
\lesssim_{N,\vp}
m^{-2-(N+1)/2}\lambda + m^{1/2}.
\end{align*}
Therefore,
we obtain 
\begin{align*}
\abs{K_{\ap,1}[\vp](m,\lambda)-\n_1^{(\ap)}(0)}
\lesssim_{N,\vp}
m^{-2-(N+1)/2}\lambda + m^{1/2}
\end{align*} 
and 
\begin{align*}
\abs{K_{\alpha,2}[\vp](m,\lambda)-\n_2^{(\alpha)}(0)}
\lesssim_{N,\vp}
m^{-2-(N+1)/2}\lambda + m^{1/2}.
\end{align*} 
As in the proof of the second inequality,
we have 
\begin{align*}
\abs{K_{(0,0,\alpha_1,\alpha_2),1}[\vp](m,\lambda)
-\n_1^{(0,0,\alpha_1,\alpha_2)}(0)}
\lesssim_{N,\vp}
m^{-2-(N+1)/2}\lambda + m^{1/2}.
\end{align*} 
Hence the desired estimate (\ref{est:thm:main:1}) holds true 
if $\alpha\in \Ajk{1}{5}$ or $\alpha\in \Ajk{1}{6}$, 
which completes the proof of Theorem \ref{thm:main:1}.

\begin{rem}\label{rem:prf:thm:main:1}
We now explain 
why we set Hypotheses (H3)--(H5).
\begin{enumerate}[(1)]
  \item
In Cases 2--5 of the above proof,
we use Proposition \ref{prop:RL:5}(4).
In particular, 
one has to assume that $N$ is odd.  
Therefore, 
at this stage, 
(H5) is a necessary condition
to prove the following statement (P):
\begin{enumerate}[(P)]
  \item
if $\nn^{(\beta)}(0)$ is known 
for any $\beta\in \Z^4$ with $|\beta|\le N-4$, 
then we have the reconstruction formula for $\nn^{(\alpha)}(0)$ 
for any $\alpha\in \Z^4$ with $|\alpha|=N$.
\end{enumerate}
  \item 
In Case 5,
if we also calculate $b_\alpha:=J_{\alpha,1}[\vp](m)$,
then we have 
\begin{align}\label{est:rem:prf:thm:main:1}
\abs{
\Miii{b_\alpha/d}{b_{\ap }/d}{b_{\app}/d}
-
\Miii{N & \alpha_1 & -\alpha_2 }{
1 & \alpha_2+1 & \alpha_2}{-1 & \alpha_1 & \alpha_1+1}
\Miii{\n_2^{(\alpha)}(0)}{\n_1^{(\ap)}(0)}{\n_1^{(\app)}(0)}
}
 \lesssim_{N,\vp} m^{1/2}.
\end{align} 
Unfortunately,
this $3\times 3$ matrix is singular.
Hence, 
one has to introduce some hypothesis on $\nn$.
In the present paper, 
based on (H2),
we set (H3).
Once (H3) is assumed,
the above calculation of $b_\alpha$ is no longer needed.
  \item 
Let us consider Case $\Ajk{1}{7}$.
In fact, 
we have (\ref{est:rem:prf:thm:main:1}) or its equivalent.
However, 
unlike Case 5,
we can not uniquely determine the three derivatives,
even if we set an $\Ajk{1}{7}$-version of (H3).
In the present paper, 
in view of this difficulty, 
we assume (H4) to prove (P).
\end{enumerate}
\end{rem}

\section{Proof of Theorem \ref{thm:main:2}}\label{sec:prf:main:2}
In order to prove the second main result, 
we first introduce a property of the function $W_\alpha^\A[\vp]$.

\begin{lem}\label{lem:W-P}
Let $N\in \N$ satisfy $N\ge 9$.
Then we have the following property (W$_N$):
\begin{enumerate}[(W$_N$)]
  \item
For any $\vp\in B_1 H^1$ and 
$\alpha \in \Z^4$ with $|\alpha|=N$, 
there exist functions 
$\w_{\delta,\epsilon}^{\vp,\alpha}$ 
$(\delta,\epsilon\in \Z^4$ with 
$5\le |\delta|\le N-4$ and $|\epsilon|\le k_\alpha +1)$ 
in $L^1X$ 
such that 
\begin{align*}
\max_{\substack{
5\le |\delta|   \le N-4
\\
|\epsilon| \le k_\alpha +1
}}
\nr{\w_{\delta,\epsilon}^{\vp,\alpha}}_{L^1X} 
\lesssim_N 1,
\end{align*} 
and that  
for any mapping $\A:\Z^4 \to \C^2$, 
\begin{align*}
W_{\alpha}^\A[\vp]
=
\sum_{\substack{
5\le |\delta|   \le N-4
\\
|\epsilon| \le k_\alpha +1
}}
\brad{ \A(\delta) }^\epsilon \w_{\delta,\epsilon}^{\vp,\alpha}.
\end{align*} 
\end{enumerate}
\end{lem}

\begin{proof}
Set $\vp\in B_1 H^1$ and 
$\alpha \in \Z^4$ with $|\alpha|=N$.
Let $\A$ be a mapping from $\Z^4$ to $\C^2$.
For convenience, 
we write the definition of 
$V_{\alpha^\prime,\beta^\prime}^\A[\vp]$ 
($\alpha^\prime,\beta^\prime \in \Z^4$):
\begin{align*}
V_{\alpha^\prime,\beta^\prime}^\A[\vp]
=
\sum_{\gamma\le \alpha^\prime}
\frac{\alpha^\prime!}{\gamma! (\alpha^\prime-\gamma)!}
\brad{R_\vp,L_\vp}^{\alpha^\prime-\gamma}
\brad{iM_\vp,iM_\vp}^{\gamma}
\A(\beta^\prime+\alpha^\prime-\gamma+\widetilde{\gamma}).
\end{align*} 
Note that by Proposition \ref{prop:St}(2),  
we have the following properties: 
\begin{enumerate}[(V-1)]
  \item 
$V_{\mathbf{0},\beta^\prime}^\A[\vp]=\A(\beta^\prime)$.
  \item 
If $|\alpha^\prime|>0$, 
then there exist functions 
$
v_{\gamma,1}^{\vp,\alpha^\prime,\beta^\prime},\cdots,
v_{\gamma,|\alpha^\prime|}^{\vp,\alpha^\prime,\beta^\prime}
$ 
($\gamma\in \Z^4$ with $|\gamma|=|\alpha^\prime+\beta^\prime|$) 
which are in $Z$, 
and independent of $\A$  
such that 
\begin{align*}
\max_{\substack{
|\gamma|=|\alpha^\prime+\beta^\prime|
\\
\ell=1,\cdots,|\alpha^\prime|
}
}
\nr{v_{\gamma,\ell}^{\vp,\alpha^\prime,\beta^\prime}}_Z
\lesssim_{|\alpha^\prime+\beta^\prime|} 1,
\end{align*} 
and that 
\begin{align*}
V_{\alpha^\prime,\beta^\prime}^\A[\vp]
=
\sum_{|\gamma|=|\alpha^\prime+\beta^\prime|}
v_{\gamma,1}^{\vp,\alpha^\prime,\beta^\prime}
\cdots
v_{\gamma,|\alpha^\prime|}^{\vp,\alpha^\prime,\beta^\prime}
\A(\gamma) .
\end{align*}
\end{enumerate}

Suppose first that $9\le N <13$.
Then we have $k_\alpha=1$.
By the definition of $W_\alpha^\A[\vp]$, 
we obtain 
\begin{align*}
  W_\alpha^\A[\vp]
&=
\sum_{|\beta|=1}
\frac{\alpha!}{\beta!}
\sum_{\substack{
\alpha(0)+\alpha(1)=\alpha
\\
4 \le |\alpha(0)| \le 7
\\
5 \le |\alpha(1)| \le 8
}
}
\frac{
\cc_{\beta,1} \braa{ \Gs \wW_{\alpha(1)}^\A[\vp] }
}{\alpha(1)!}
\frac{
V_{\alpha(0),\beta}^\A[\vp]
}{\alpha(0)!}
\\
&=
\sum_{|\beta|=1}
\frac{\alpha!}{\beta!}
\sum_{\substack{
\alpha(0)+\alpha(1)=\alpha
\\
4 \le |\alpha(0)| \le 7
\\
5 \le |\alpha(1)| \le 8
}
}
\frac{
\cc_{\beta,1} \braa{ \Gs V_{\alpha(1)}^\A[\vp] }
}{\alpha(1)!}
\frac{
V_{\alpha(0),\beta}^\A[\vp]
}{\alpha(0)!}.
\end{align*}  
Since 
multi-indices $\alpha(0),\alpha(1)$ and $\beta$ 
in the above identity satisfy 
$5\le |\alpha(1)|,|\alpha(0)|+|\beta|\le N-4$, 
by Corollary \ref{cor:prop:max},
Proposition \ref{prop:St}(3),
and (V-2),
we obtain (W$_N$).

Suppose next that $N \ge 13$, 
and that (W$_n$) holds true for any $n\in \N$ with $9\le n < N$.
By the definition of $W_\alpha^\A[\vp]$,
we obtain 
\begin{align}\label{def:2:prf:lem:W-P}
  W_\alpha^\A[\vp]
=
\sum_{1\le |\beta| \le k_\alpha}
\ssum
\frac{\alpha!}{\beta! \alpha(0)! \cdots \alpha(|\beta|)!}
\braa{
\prod_{\ell=1}^{|\beta|}
\cc_{\beta,\ell} \braa{ \Gs \wW_{\alpha(\ell)}^\A[\vp] }
}
V_{\alpha(0),\beta}^\A[\vp].
\end{align}
As for functions $\wW_{\alpha(\ell)}^\A[\vp]$ 
and $V_{\alpha(0),\beta}^\A[\vp]$  
in this identity, 
we have 
\[
\wW_{\alpha(\ell)}^\A[\vp]
=
V_{\alpha(\ell)}^\A[\vp]+W_{\alpha(\ell)}^\A[\vp],  
\]
$|\alpha(0)+\beta| \le N-4$, 
and 
$5\le |\alpha(\ell)|\le N-4$ (see Remark \ref{rem:isp-notation}(7)).
In particular, 
by Assumption (W$_{|\alpha(\ell)|}$), 
we can express 
\begin{align*}
\wW_{\alpha(\ell)}^\A[\vp]
=
\sum_{\substack{
5\le |\delta(\ell)|   \le N-4
\\
|\epsilon(\ell)| \le k_{\alpha(\ell)} +1
}}
\brad{ \A(\delta(\ell)) }^{\epsilon(\ell)} 
\widetilde{\w}_{\delta(\ell),\epsilon(\ell)}^{\vp,\alpha(\ell)},
\end{align*} 
where 
$\widetilde{\w}_{\delta(\ell),\epsilon(\ell)}^{\vp,\alpha(\ell)}$ 
($\delta(\ell),\epsilon(\ell)\in \Z^4$ with 
$5\le |\delta(\ell)|\le N-4$ and $|\epsilon(\ell)|\le k_{\alpha(\ell)} +1)$) 
are in $L^1X$ 
and independent of $\A$.
Therefore, 
(W$_N$) follows from 
(\ref{def:2:prf:lem:W-P}),
Corollary \ref{cor:prop:max},
Proposition \ref{prop:St}(3),
(V-1), 
(V-2), 
and 
\begin{align}\label{est:prf:lem:W-P}
\braa{k_{\alpha(1)}+1}+\cdots+\braa{k_{\alpha(|\beta|)}+1}+1
\le k_\alpha +1.
\end{align} 
Here, 
$\beta,\alpha(0),\alpha(1),\cdots,\alpha(|\beta|)$ 
are multi-indices appearing in (\ref{def:2:prf:lem:W-P}).
We finally show (\ref{est:prf:lem:W-P}).
Set $k=\braa{k_{\alpha(1)}+1}+\cdots+\braa{k_{\alpha(|\beta|)}+1}$.
Then we have 
\begin{align*}
(5-k)_+ + 5k
&\le 
4k+5
\\
&\le 
\braa{ 4k_{\alpha(1)}+5 }
+\cdots+
\braa{ 4k_{\alpha(|\beta|)}+5 }
+5-|\beta|
\\
&\le
|\alpha(1)+\cdots +\alpha(|\beta|)|+5-|\beta|
\le
|\alpha|+5-|\alpha(0)|-|\beta|
\le |\alpha|.
\end{align*}
By the definition of $k_\alpha$,
we obtain 
$k\le k_\alpha$. 
Hence (\ref{est:prf:lem:W-P}) holds true.
\end{proof}

We immediately see from 
Lemma \ref{lem:W-P} that 
the following proposition holds true:

\begin{prop}\label{prop:W-W}
For any mappings 
$\A,\B:\Z^4 \to \C^2$,
$\vp\in B_1 H^1$, 
and $\alpha\in \Z^4$ with $|\alpha|\ge 9$, 
we have 
\begin{align*}
&
\nr{
W_\alpha^\A[\vp]
-
W_\alpha^\B[\vp]
}_{L^1X}
\\
&\lesssim_{|\alpha|}
\brab{
\max_{5\le |\delta|\le |\alpha|-4}
\braa{
1+\abs{\A(\delta)}+\abs{\B(\delta)}
}
}^{k_\alpha}
\brab{
\max_{5\le |\delta|\le |\alpha|-4}
\abs{\A(\delta)-\B(\delta)}
}.
\end{align*} 
\end{prop}

We are ready to prove the second main theorem.
In fact, 
the proof is easy because of Theorem \ref{thm:main:1} 
and Proposition \ref{prop:W-W}.
\begin{proof}[Proof of Theorem \ref{thm:main:2}]
Suppose (H1) and (H3)--(H5).
Choose an even function 
$\vp \in \Sc \cap B_{\eta_0}H^1$ 
so that $\supp\vp \subset [-1,1]$, 
and that $\vp>0$ on $(-1,1)$.
Fix $N\in \Z$ with $N\ge 5$.
Let $\lambda$ satisfy 
\begin{align*}
0<\lambda<\braa{\frac{1}{N+1}}^{1/p(N)}.
\end{align*} 
It suffices to prove the estimate 
\begin{align}\label{est:prf:thm:main:2}
\abs{
\mZ_n[\lambda](\alpha)
-
\nn^{(\alpha)}(0)
}
\lesssim_{N,\vp}
\lambda^{q(n)} 
\end{align} 
for any $n=5,6,\cdots,N$ and $\alpha\in \Z^4$ with $|\alpha|=n$   
because the desired estimate in Theorem \ref{thm:main:2} 
holds true when $n=N$ and $r=\lambda^{p(N)}$.

We prove (\ref{est:prf:thm:main:2}) 
by induction on $n$ up to $N$.
If $5\le N<9$ and $5\le n \le N$, 
then Theorem \ref{thm:main:1} 
(see also Remark \ref{rem:thm:main:1}(1)) 
directly implies (\ref{est:prf:thm:main:2})  
for any $\alpha\in \Z^4$ with $|\alpha|=n$. 
In fact, 
we have 
\begin{align*}
\abs{
\mZ_n[\lambda](\alpha)
-
\nn^{(\alpha)}(0)
}
&=
\max_{j=1,2}
\abs{
K_{\alpha,j}(\lambda^{2/(n+6)}/2,\lambda)
-
\n_j^{(\alpha)}(0)
}
\lesssim_{N,\vp} \lambda^{1/(n+6)}=\lambda^{q(n)}.
\end{align*} 
Henceforth, 
we suppose that $N\ge 9$, 
and that 
(\ref{est:prf:thm:main:2}) holds true 
for any $n=5,6,\cdots,N-1$ 
and $\alpha\in \Z^4$ with $|\alpha|=n$.
We now choose $\alpha\in \Z^4$ so that $|\alpha|=N$.
Since 
\begin{align*}
\braa{\lambda^{2q(N)}}^{-2-(N+1)/2}
\lambda^{p(N)}
+
\braa{\lambda^{2q(N)}}^{1/2}
=
2\lambda^{q(N)}, 
\end{align*} 
we see from Theorem \ref{thm:main:1} that 
\begin{align*}
&
\abs{
\mZ_N[\lambda](\alpha)
-
\nn^{(\alpha)}(0)
}
\\
&
\le 
\abs{
K_{\alpha,j}^{\mZ_{N-4}[\lambda]}[\vp](\lambda^{2q(N)}/2,\lambda^{p(N)})
-
K_{\alpha,j}[\vp](\lambda^{2q(N)}/2,\lambda^{p(N)})
}
\\
&\quad +
\abs{
K_{\alpha,j}[\vp](\lambda^{2q(N)}/2,\lambda^{p(N)})
-
\n_j^{(\alpha)}(0)
}
\\
&\lesssim_{N,\vp}
\abs{
K_{\alpha,j}^{\mZ_{N-4}[\lambda]}[\vp](\lambda^{2q(N)}/2,\lambda^{p(N)})
-
K_{\alpha,j}[\vp](\lambda^{2q(N)}/2,\lambda^{p(N)})
}
+
\lambda^{q(N)} 
\end{align*} 
for any $j=1,2$.
Once we obtain 
\begin{align}\label{est:2:prf:thm:main:2}
\abs{
I_{\alpha,j}^{\mZ_{N-4}[\lambda]}[\vp](\lambda^{2q(N)}/2,\lambda^{p(N)})
-
I_{\alpha,j}[\vp](\lambda^{2q(N)}/2,\lambda^{p(N)})
}
\lesssim_{N,\vp}
\lambda^{q(N)}, 
\quad j=1,2, 
\end{align}
by Proposition \ref{prop:RL:5}(3) and (4) 
we have (\ref{est:prf:thm:main:2}) with $n=N$, 
which completes the proof. 
Define $m=\lambda^{2q(N)}$.
It follows from the definition of $I_{\alpha,j}^\A$ 
introduced in Subsection \ref{sebsec:main} that 
\begin{align}
&
\abs{
I_{\alpha,j}^{\mZ_{N-4}[\lambda]}[\vp](\lambda^{2q(N)}/2,\lambda^{p(N)})
-
I_{\alpha,j}[\vp](\lambda^{2q(N)}/2,\lambda^{p(N)})
}
\nonumber
\\
&\lesssim_N 
m^{-2-(N+1)/2}
\nr{
W_\alpha^{\mZ_{N-4}[\lambda]}[m^{1/2}D_m \vp]
-
W_\alpha [m^{1/2}D_m \vp]
}_{L^1X}
\nr{m^{1/2}D_m \vp}
\nonumber
\\
&\lesssim_N 
\lambda^{-(N+5) q(N)}
\nr{
W_\alpha^{\mZ_{N-4}[\lambda]}[m^{1/2}D_m \vp]
-
W_\alpha [m^{1/2}D_m \vp]
}_{L^1X}
\label{est:3:prf:thm:main:2}
\end{align} 
for any $j=1,2$.
Here we have used the estimate $\nr{m^{1/2}D_m \vp}\le 1$. 
Furthermore,  
we can apply Proposition \ref{prop:W-W}:  
\begin{align}
&
\nr{
W_\alpha^{\mZ_{N-4}[\lambda]}[m^{1/2}D_m \vp]
-
W_\alpha [m^{1/2}D_m \vp]
}_{L^1X}
\nonumber
\\
&\lesssim_N
\brab{
\max_{5\le |\delta|\le N-4}
\braa{
1+\abs{\mZ_{N-4}[\lambda](\delta)}+\abs{\nn^{(\delta)}(0)}
}
}^{k_\alpha}
\brab{
\max_{5\le |\delta|\le N-4}
\abs{\mZ_{N-4}[\lambda](\delta)-\nn^{(\delta)}(0)}
}.
\label{est:4:prf:thm:main:2}
\end{align} 
It follows from the definition of $\mZ_{N-4}[\lambda]$ 
and the induction hypothesis that 
\begin{align*}
\max_{5\le |\delta|\le N-4}
\abs{\mZ_{N-4}[\lambda](\delta)-\nn^{(\delta)}(0)}
&=
\max_{5\le |\delta|\le N-4}
\abs{\mZ_{|\delta|}[\lambda](\delta)-\nn^{(\delta)}(0)}
\\
\lesssim_{N,\vp} 
\max_{5\le |\delta|\le N-4} \lambda^{q(|\delta|)}
\le 
\lambda^{q(N-4)}.
\end{align*} 
We see from (\ref{est:3:prf:thm:main:2})  
and (\ref{est:4:prf:thm:main:2}) that  
\begin{align*}
\abs{
I_{\alpha,j}^{\mZ_{N-4}[\lambda]}[\vp](\lambda^{2q(N)}/2,\lambda^{p(N)})
-
I_{\alpha,j}[\vp](\lambda^{2q(N)}/2,\lambda^{p(N)})
}
&
\lesssim_{N,\vp} 
\lambda^{-(N+5) q(N)+q(N-4)} 
\\
&
=
\lambda^{q(N)},
\quad j=1,2.
\end{align*} 
Hence (\ref{est:2:prf:thm:main:2}) holds true.
\end{proof}


\end{document}